\documentclass{amsart}

\usepackage{mathtools}
\usepackage{amsmath}
\usepackage[shortlabels]{enumitem}
\usepackage{url}
\usepackage{ bbold }
\usepackage[margin=3cm]{geometry}
\usepackage{float}
\usepackage{ mathrsfs }
\DeclarePairedDelimiter\ceil{\lceil}{\rceil}
\DeclarePairedDelimiter\floor{\lfloor}{\rfloor}

\newcommand{\R}{\mathbb{R}}
\newcommand{\N}{\mathbb{N}}
\newcommand{\C}{\mathbb{C}}
\newcommand{\Z}{\mathbb{Z}}

\newcommand{\D}{\mathbb{D}}
\newcommand{\T}{\mathbb{T}}
\newcommand{\E}{\mathbb{E}}

\renewcommand{\Im}{\operatorname{Im}}
\newcommand{\tr}{\operatorname{Tr}}
\renewcommand{\d}{\mathrm{d}}

\newtheorem{thm}{Theorem}[section]
\newtheorem{prop}[thm]{Proposition}
\newtheorem{lemma}[thm]{Lemma}
\newtheorem{cor}[thm]{Corollary}

\theoremstyle{definition}

\theoremstyle{remark}
\newtheorem{remark}[thm]{Remark}

\begin{document}

\title[Traces of orthogonal and symplectic
matrices]{Multivariate normal approximation for traces of orthogonal and symplectic matrices}

\date{\today}
\author{Klara Courteaut\textsuperscript{*}} 
\author{Kurt Johansson\textsuperscript{$\dagger$}}
\thanks{\textsuperscript{*}{Department of Mathematics, KTH Royal Institute of Technology, klaraco@kth.se}}
\thanks{\textsuperscript{$\dagger$}Department of Mathematics, KTH Royal Institute of Technology, kurtj@kth.se}

\begin{abstract}
We show that the distance in total variation between $(\tr U, \frac{1}{\sqrt{2}}\tr U^2, \cdots, \frac{1}{\sqrt{m}}\tr U^m)$ and a real Gaussian vector, where $U$ is a Haar distributed orthogonal or symplectic matrix of size $2n$ or $2n+1$, is bounded by $\Gamma(2\frac{n}{m}+1)^{-\frac{1}{2}}$ times a correction. The correction term is explicit and holds for all $n\geq m^4$, for $m$ sufficiently large. For $n\geq m^3$ we obtain the bound $(\frac{n}{m})^{-c\sqrt{\frac{n}{m}}}$ with an explicit constant $c$. Our method of proof is based on an identity of Toeplitz+Hankel determinants due to Basor and Ehrhardt, see \cite{BE}, which is also used to compute the joint moments of the traces. 
\end{abstract}

\maketitle

\setcounter{section}{-1}

\section{Introduction}
\label{introduction}

\setcounter{equation}{0}
\setcounter{thm}{0}

Let $U$ be an element of the orthogonal, unitary or symplectic groups distributed according to normalized Haar measure. In \cite{DS} Diaconis and Shahshahani showed that the joint moments of $\tr U, \tr U^2$, $\cdots$, $\tr U^m$ up to a certain order are equal to those of $m$ independent (complex) Gaussians if the matrices are orthogonal or symplectic (unitary). An immediate consequence is the convergence of the vector $(\tr U, \tr U^2, \cdots \tr U^m)$ to a Gaussian vector as the size of the matrix goes to infinity, and a natural question is its rate of convergence, since the moments are equal to high order. A first answer to this question was given by Stein in \cite{Stein} who obtained a super-polynomial rate of convergence in total variation for a single trace in the case of the orthogonal group. Later the second author of this paper considered linear combinations of the above traces of powers of matrices and showed in \cite{J} that the total variation between those and a Gaussian is bounded by $C_1n^{-\delta_1 n}$ for the unitary case and $C_2e^{-\delta_2 n}$ for the orthogonal/symplectic case, for some non-explicit constants $C_1, C_2$ and $\delta_1, \delta_2$. This result was recently extended to the multivariate case by the second author and Lambert in \cite{JL} where they also allowed the highest power $m$ to increase with the size of the matrix $n$ and kept track of the constants. They proved that the rate of convergence, again in total variation, is bounded by $\Gamma(\frac{n}{m}+1)^{-1}$ times a correction term, provided $m$ grows slower than approximately $\sqrt{n}$. This paper is devoted to the same problem but for orthogonal and symplectic matrices. It improves considerably the result in \cite{J}. We show that if $n\geq m^4$ and $m$ is sufficiently large (see Corollary \ref{approx1}), the total variation is bounded by
\[ 16 m^\frac{3}{2} \sqrt{\Omega_m} (24 nm \log N)^{\frac{m}{4}} \frac{(e^{3/2}(\log m +1))^{N}}{\sqrt{N}\sqrt{\Gamma(2N+1)}} \]
where $\Omega_m = \pi^\frac{m}{2}/\Gamma(\frac{m}{2}+1)$ is the volume of the unit $m$-ball. If $m$ is smaller than what's assumed in the corollary, or if we only assume $n\geq m^3$, we loose the factor $N^{-N}$ from the Gamma function which is replaced by $N^{-c_1N}$ or $N^{-c_2\sqrt{N}}$ for some constants $c_1$ and $c_2$ (see Corollary \ref{approx2} and \ref{approx3}). We also mention that the case of a single power (i.e. $\tr U^k$, $k\geq 1$) for any of the above mentioned groups allows more precise estimates and is considered in a forthcoming paper \cite{CJK}. Another related problem is the rate of convergence in Wasserstein distance which was studied by Döbler and Stolz in \cite{DoblerStolz} for the multivariate case and the unitary, special orthogonal, and unitary symplectic groups. \\

In the following we let $U$ denote a random matrix drawn from either the orthogonal or symplectic group under normalized Haar measure. We consider the vector $\mathbf{X} = (X_1,X_2,\dots,X_m)$, where $X_k = \frac{1}{\sqrt{k}}(\tr U^k-\E_{G(n)} [ \tr U^k])$ and we denote by $F_{n,m}^{a,b}$ its characteristic function, where $m$ is the length of the vector, $n$ determines the size of the matrix, and $a,b$ specify which group the matrix belongs to. The values of $a$ and $b$ appear in the joint eigenvalue density of the matrices, which is given by
\begin{align}\label{density}
\tilde{\rho}_n^{a,b}(x) = \frac{1}{Z_n^{a,b}}\prod_{1\leq j\leq n}(1-x_j)^a(1+x_j)^b\prod_{1\leq j<k\leq n}(x_j-x_k)^2    
\end{align}
on $[-1,1]^n$, where $Z_n^{a,b}$ is the normalization constant, $Z_n^{a,b} = (\pi^n n!)/2^{n^2-(1-a-b)n+\mathbb{1}\{a,b<0\}}$.  If we set $(a,b)=(1/2,1/2)$ we obtain the eigenvalue density of $Sp(2n)$ and $O(2n)^-$, and if $(a,b)=(-1/2,-1/2)$, $(a,b)=(-1/2,1/2)$ and $(a,b)=(1/2,-1/2)$ we get the density of $O(2n)^+$, $O(2n+1)^-$ and $O(2n+1)^+$ respectively. For $O(2n)^-$, the probability density is actually $\tilde{\rho}_{n-1}^{1/2,1/2}$, i.e. $n$ is replaced by $n-1$.
To keep notation simple we will sometimes replace $a,b$ with their respective sign.
Observe that the density is supported on  $[-1,1]^n$ but the eigenvalues of the random matrices all lie on the unit circle. The reason is that all eigenvalues except for 1 and -1 occur in conjugate pairs, so (\ref{density}) is obtained by making the change of variables $x_j=\cos(\theta_j)$ in the following eigenangle densities supported on $[0,\pi)$: 
\begin{align}\label{density2}
\rho_n^{--}(\theta)&=\frac{2^{(n-1)^2}}{n!\pi^n} \prod_{1\leq j<k\leq n}(\cos\theta_j-\cos\theta_k)^2\\  \nonumber
\rho_n^{++}(\theta)&=\frac{2^{n^2}}{n!\pi^n}\prod_{1\leq j\leq n}\sin^2\theta_j \prod_{1\leq j<k\leq n}(\cos\theta_j-\cos\theta_k)^2 \\ \nonumber
\rho_n^{-+}(\theta)&=\frac{2^{n^2}}{n!\pi^n}\prod_{1\leq j\leq n}\cos^2\frac{\theta_j}{2} \prod_{1\leq j<k\leq n}(\cos\theta_j-\cos\theta_k)^2 \\ \nonumber
\rho_n^{+-}(\theta)&=\frac{2^{n^2}}{n!\pi^n}\prod_{1\leq j\leq n}\sin^2\frac{\theta_j}{2}\prod_{1\leq j<k\leq n}(\cos\theta_j-\cos\theta_k)^2. 
\end{align}
These identities are due to H. Weyl and often called the Weyl integration formula, see \cite{Meckes}. Note also that there are deterministic eigenvalues: $O(2n+1)^+$ must have an eigenvalue at $1$, $O(2n+1)^-$ at $-1$ and $O(2n)^-$ at both $1$ and $-1$, for their determinant to have the correct value; this and the fact that the eigenvalues come in conjugate pairs explain why there are only $n$ variables in the eigenvalue density of $O(2n+1)^\pm$ and $n-1$ in that of $O(2n)^-$. These deterministic eigenvalues do not appear in the above joint eigenvalue densities but they also do not affect the random vector $X$ since it is centered. Thus in the following, we will only consider random eigenvalues and write $\tr U$ for their sum, disregarding the possible eigenvalues at $\pm1$.

We will use both (\ref{density}) and (\ref{density2}) for the eigenvalue densities so to differentiate them we will denote by $\E_n^{a,b}$ the expected value with respect to (\ref{density}) and by $\E_{G(n)}$ the expected value with respect to (\ref{density2}), where $G(n)$ denotes either $O(2n)^+$, $O(2n)^-$, $O(2n+1)^+$, $O(2n+1)^-$ or $Sp(2n)$. The characteristic function $F_{n,m}^{a,b}$ is therefore equal to
\begin{align*}
F_{n,m}^{a,b}(\xi) &= \E_{G(n)}\Big[\exp{\Big(i\sum_{1\leq k\leq m} \frac{\xi_k}{\sqrt{k}}(\tr U^k}-\E_{G(n)} [ \tr U^k])\Big)\Big] \\
&= \E_n^{a,b}\Big[ \prod_{1\leq j\leq n} \exp{\Big(i\sum_{1\leq k\leq m} \frac{\xi_k}{\sqrt{k}}\Big(2\mathrm{T_k}(x_j)-\frac{\E_{G(n)} [ \tr U^k]}{n}\Big)}\Big)\Big]
\end{align*}
where $\xi=(\xi_1,\xi_2,\cdots,\xi_m)\in\R^m$ and $T_k$ is the $k$th Chebyshev polynomial. The mean value $\E_{G(n)} [ \tr U^k]$, both including and excluding the deterministic eigenvalues, is given in Proposition \ref{moments} and the following remark. We also introduce the functions 
\begin{align}\label{f}
f(x)=\sum_{1\leq k\leq m} \frac{\xi_k}{\sqrt{k}}\Big(2\mathrm{T_k}(x)-\frac{\E_{G(n)} [ \tr U^k]}{n}\Big) , \quad x\in [-1,1]    
\end{align}
and
\begin{equation}\label{g}
g(\theta)= f(\cos(\theta)) = \sum_{1\leq k\leq m} \frac{\xi_k}{\sqrt{k}} \Big(2\cos{k\theta}-\frac{\E_{G(n)} [ \tr U^k]}{n}\Big),\quad \theta\in[0,\pi)    
\end{equation}
and write $\tr g(U)=\sum_{1\leq j\leq n}g(\theta_j)$. Now we have
$$ F_{n,m}^{a,b}(\xi)= \E_{G(n)}[e^{i\tr g(U)}] = \E_n^{a,b}[\prod_{1\leq j\leq n} e^{if(x_j)}].$$
Finally we let $\mathscr{P}_{n,m}^{a,b}$ be the probability density of the random vector $\mathbf{X}$ and $\Psi_{n,m}(x)= \frac{e^{-\|x\|^2/2}}{(2\pi)^{m/2}}$ that of a standard normal vector. We denote by $\Delta_{n,m}^{(2)}$ the $L_2$ distance between $\mathscr{P}_{n,m}^{a,b}$ and $\Psi_{n,m}$ and by $\Delta_{n,m}^{(1)}$ the $L_1$ distance (i.e. the total variation). We obtain the following bounds on $\Delta_{n,m}^{(2)}$.

\begin{thm}\label{L2}
Assume $n\geq m^3$, $m\geq3$. For any pair $(a,b)= (\pm 1/2,\pm 1/2)$,
\begin{multline}\label{maineq}
\Delta_{n,m}^{(2)} \leq \sqrt{\Omega_m}N^{m/2}\left[ \frac{16}{15}e^{13/24}(e^{9/8}+1)  \frac{m^{3/2}}{N^{\frac{m+1}{2}}} \Big(\frac{m}{2}\Big)^{\frac{m}{4}}\frac{(e^{3/2}(\log m +1))^{N}}{\sqrt{\Gamma(2N+1)}}\right.\\
+ \sqrt{3} (2e)^{4m^2} \Big(\frac{\sqrt{c_3(m)}m}{(2\pi n)^{1/N}}\Big)^m \exp\Big(-\frac{\left(1- c_1(m)\right)^2 n^2}{3c_2(m)(m+1)^{8/3}(\log m+1)}\Big)\\
+  m^\frac{m}{2} \exp\Big(-\frac{\left(1- c_1(m)\right)^2 N^2}{4c_2(m)(m+1)^{8/3}(\log m+1)^2}\Big) \left.+ \frac{\sqrt{m}}{(2\sqrt{\log m +1})^\frac{m-2}{2}N} e^{ -\frac{N^2}{8(\log m+1)}} \right]
\end{multline}
where $c_1(m)$, $c_2(m)$ and $c_3(m)$ are defined in (\ref{c1}), (\ref{c2}) and (\ref{c3}).
\end{thm}

As a consequence we are able to derive a bound on the total variation.

\begin{thm}\label{L1}
Assume $m\geq4$. For any pair $(a,b)= (\pm 1/2,\pm 1/2)$,
\begin{align*}
    \Delta_{n,m}^{(1)} \leq 2(48 m \log \Delta_{n,m}^{a,b\ -1})^{\frac{m}{4}}\Delta_{n,m}^{(2)}
\end{align*}
for $n\geq m^4$, provided $\Delta_{n,m}^{(2)} \leq 3m(2\sqrt{3e}m)^{-\frac{m}{2}}$ and 
\begin{align*}
    \Delta_{n,m}^{(1)} \leq 2(80 m \log \Delta_{n,m}^{a,b\ -1})^{\frac{m}{4}}\Delta_{n,m}^{(2)}
\end{align*}
for $n\geq m^3$, provided $\Delta_{n,m}^{(2)} \leq 2.5m(2\sqrt{5e}m)^{-\frac{m}{2}}$.
\end{thm}

We can simplify these results by considering special cases of $m$ and $n$. For example,

\begin{cor}\label{approx1}
If $m$, $n$ satisfy the conditions in one column of the following table
\begin{table}[H]
\begin{tabular}{|l|l|l|l|l|l|l|l|} \hline
$n\geq$ & $m^4$ & $m^5$ & $m^6$ & $m^7$ & $m^8$ & $m^9$ & $m^{10}$   \\ \hline
$m\geq$ & $10^{19}$ & $1140$ & $34$ & $11$ & $6$ & $5$ & $4$  \\ \hline
\end{tabular}
\end{table}
\noindent then,
\[\Delta_{n,m}^{(2)}\leq 8m^\frac{3}{2} \sqrt{\Omega_m} \Big(\frac{m}{2}\Big)^{\frac{m}{4}} \frac{(e^{3/2}(\log m +1))^{N}}{\sqrt{N}\sqrt{\Gamma(2N+1)}} \]
and 
\[ \Delta_{n,m}^{(1)}\leq 16 m^\frac{3}{2} \sqrt{\Omega_m} (24 nm \log N)^{\frac{m}{4}} \frac{(e^{3/2}(\log m +1))^{N}}{\sqrt{N}\sqrt{\Gamma(2N+1)}}. \]
\end{cor}
See also corollaries \ref{approx2} and \ref{approx3} for other conditions on $m$ and $n$. We do not know how fast $m$ is allowed to grow relatively to $n$ to obtain a fast rate of convergence. In \cite{DS} it is
suggested that there is some analogy between the present problem and the fast rate of convergence of the vector of cycle lengths $(C_1,\dots, C_m)$
in a uniform random permutation to a vector of Poisson random variables. The fast rate of convergence in that problem has been proved by Arratia and Tavaré \cite{RS}. It holds if $m/n$ goes to zero but if $m$ is a small multiple of $n$ then the $m$-tuple of cycle lengths is not approximated
by the $m$-tuple of Poisson random variables. If we believe in the analogy between the two problems we could conjecture that the fast rate of convergence for $m$-tuples of traces holds if $m/n$ goes to zero, but not if $m$ is of order $n$. \\

The paper is organised as follows: first we present some general facts about integrals over the orthogonal/symplectic groups: an analogue of Heine's identity and a result of Basor and Ehrhardt \cite{BE} expressing Toeplitz+Hankel determinants using Fredholm determinants. Combined they give a new proof of the moment identities of Diaconis and Shahshahani \cite{DS}. They also give our estimates on the characteristic function for what we call the small regime of $\xi$, which we give in the second section. The intermediate and large regimes are treated in the third and fourth sections by making a certain change of variables in the integral expression of the characteristic function, a method first introduced in \cite{Jthesis}. We also reuse \cite{BE} and the results of \cite{Chahkiev}. In the last section we gather all our estimates to bound $\Delta_{n,m}^{(2)}$, the $L_2$ distance between $\mathscr{P}_{n,m}^{a,b}$ and $\Psi_{n,m}$, via Plancherel's theorem. We then use \cite{BE} one more time to obtain tail probabilities for $\mathbf{X}$ which give us the final bound on the total variation $\Delta_{n,m}^{(1)}$.

\section{Preliminaries}

The next lemma is the orthogonal/symplectic analogue of Heine's identity which expresses Toeplitz matrices as integrals over the unitary group.

\begin{lemma}
\label{Kurt}
For any complex function $\psi$ on $[-1,1]$ we have that
\begin{align*}
& \E_n^{-+}[\prod_{j=1}^n \psi(x_j)] = \det (\hat{\phi}_{j-k}+\hat{\phi}_{j+k+1})_{0\leq i,j\leq n-1} \\
& \E_n^{+-}[\prod_{j=1}^n \psi(x_j)] = \det (\hat{\phi}_{j-k}-\hat{\phi}_{j+k+1})_{0\leq i,j\leq n-1} \\
& \E_n^{++}[\prod_{j=1}^n \psi(x_j)] = \det (\hat{\phi}_{j-k}-\hat{\phi}_{j+k+2})_{0\leq i,j\leq n-1} \\
& \E_n^{--}[\prod_{j=1}^n \psi(x_j)] = \det (\hat{\phi}_{j-k}+\hat{\phi}_{j+k})_{0\leq i,j\leq n-1} \\
\end{align*}
where $\hat{\phi}_n$ is the $n$th fourier coefficient of $\psi\circ \cos$.
\end{lemma}
\begin{proof}
The last product in the eigenvalue density (\ref{density}) is equal to the Vandermonde determinant squared, therefore
$$ \E_n^{a,b}[\prod_{j=1}^n \psi(x_j)] = \frac{1}{Z_n^{a,b}} \int_{[-1,1]^n} \prod_{j=1}^n \psi(x_j) (1-x_j)^a(1+x_j)^b \det(x_i^{j-1})^2_{1\leq i,j\leq n} \mathrm{d}^nx.$$
We can perform column operations inside the determinants and obtain
$$ C \int_{[-1,1]^n} \prod_{j=1}^n \psi(x_j) (1-x_j)^a(1+x_j)^b \det(p_{j-1}^{a,b}(x_i))^2_{1\leq i,j\leq n} \mathrm{d}^nx$$
for some constant $C$, where $\{p_j^{a,b}\}_{j=0}^n$ is any family of linearly independent polynomials such that $p_j$ has degree $j$. 
By the Cauchy-Binet identity,
$$ E_n^{a,b}[\prod_{j=1}^n \psi(x_j)] = C\cdot N! \det (\alpha_{ij})_{0\leq i,j\leq n-1}$$
where $$\alpha_{i,j}=\int_{-1}^1p_i^{a,b}(x)p_j^{a,b}(x) \psi(x)(1-x)^a(1+x)^bdx.$$ 
We recover the identities by choosing the polynomials to be normalised Jacobi polynomials, i.e.
\begin{gather*}
p_j^{--}(\cos\theta) = \sqrt{\frac{2}{\pi}}\cos j\theta, \ j\geq1; \qquad p_0^{--}(\cos\theta) = \frac{1}{\sqrt{\pi}}\\
p_j^{++}(\cos\theta) = \sqrt{\frac{2}{\pi}}\frac{\sin(j+1)\theta}{\sin\theta} \\   
p_j^{+-}(\cos\theta) = \frac{1}{\sqrt{\pi}}\frac{\sin(2j+1)\theta/2}{\sin\theta/2}; \qquad p_j^{-+}(\cos\theta) = \frac{1}{\sqrt{\pi}}\frac{\cos(2j+1)\theta/2}{\cos\theta/2} 
\end{gather*}
which are orthogonal with respect to $(1-x)^a(1+x)^b$. We see that $C=1/N!$ by letting $\psi=1$. 
\end{proof}

The Toeplitz+Hankel determinants above have a Fredholm determinant expansion, found by Basor and Ehrhardt in \cite{BE}, which we present in the next proposition. These identities are similar to the Borodin-Okounkov-Case-Geronimo formula that hold for Toeplitz determinants and will be the starting point of our analysis of the characteristic function, which ultimately will give us the bound on the total variation from a Gaussian vector. 

In the next proposition we consider functions in the Besov class $B_1^1$, i.e. functions $\omega$ on the unit circle which satisfy
\begin{equation}
\| \omega \|_{B_1^1} := \int_{-\pi}^\pi \frac{1}{y^2} \int_ {-\pi}^\pi \lvert \omega(e^{ix+iy})+\omega(e^{ix-iy})-2\omega(e^{ix}) \rvert dxdy < \infty.
\end{equation}
If $\omega \in B_1^1$ we let $\omega_+$ denote its projection on $B_{1+}^1$, the subspace of $B_1^1$ for which $\omega_k=0$ for $k<0$, and we write $\tilde{\omega}(e^{i\theta})\coloneqq \omega(e^{-i\theta})$.

\begin{prop} \cite{BE}
\label{BE}
Denote by $Q_n$ the projection operator acting on $l_2(\mathbb{N})$ that sets the first $n$ coefficients to zero, and let $H(c)$ be the Hankel operator with symbol $c\in L^\infty(\T)$, i.e. the bounded linear operator on $l_2(\mathbb{N})$ with matrix representation $H(c) = (c_{j+k+1})_{j,k=0}^\infty$, where $c_k$ is the $k$th Fourier coefficient of $c$. Assume that $b_+\in B_{1+}^1$ and set $a=a_+\tilde{a_+} = \exp(b)$ with $a_+=\exp(b_+)$, $b=b_++\tilde{b_+}$. Then,
\begin{align*}
&\det (\hat{a}_{j-k}+\hat{a}_{j+k+1})_{0\leq i,j\leq n-1}= \\
&\qquad \exp\Big(n[\log a]_0 +{\sum_{n=0}^\infty[\log a]_{2n+1}+\frac{1}{2}\sum_{n=1}^\infty n[\log a]_{n}^2}\Big) \det(1+Q_nH(a_ +^{-1}\tilde{a_+})Q_n) \\
&\det (\hat{a}_{j-k}-\hat{a}_{j+k+1})_{0\leq i,j\leq n-1}= \\
&\qquad \exp\Big({n[\log a]_0 -\sum_{n=0}^\infty[\log a]_{2n+1}+\frac{1}{2}\sum_{n=1}^\infty n[\log a]_{n}^2}\Big) \det(1-Q_nH(a_ +^{-1}\tilde{a_+})Q_n)\\
&\det (\hat{a}_{j-k}-\hat{a}_{j+k+2})_{0\leq i,j\leq n-1}= \\
&\qquad \exp\Big({n[\log a]_0 -\sum_{n=1}^\infty[\log a]_{2n}+\frac{1}{2}\sum_{n=1}^\infty n[\log a]_{n}^2}\Big) \det(1-Q_nH(t^{-1}a_ +^{-1}\tilde{a_+})Q_n)\\
&\det (\hat{a}_{j-k}+\hat{a}_{j+k})_{0\leq i,j\leq n-1}= \\
&\qquad \exp\Big({n[\log a]_0 +\sum_{n=1}^\infty[\log a]_{2n}+\frac{1}{2}\sum_{n=1}^\infty n[\log a]_{n}^2}\Big) \det(1+Q_nH(ta_ +^{-1}\tilde{a_+})Q_n)
\end{align*}
Here $[\log a]_k$ stands for the $k$th Fourier coefficient of $\log a$. The Fredholm determinants are well-defined because each Hankel operator is trace-class.
\end{prop}

As a first consequence of the above proposition we can re-derive the exact formulas of Diaconis and Shahshahani in \cite{DS} for the joint moments
of $\tr U$, $\tr U^2$, $\dots$, $\tr U^k$.

\begin{prop}\label{moments}
The moments of the traces of $\tr U$, $\tr U^2$, $\dots$, $\tr U^k$ are given by
\begin{alignat*}{2}
\E_{O(2n)^+}[\prod_{j=1}^k \tr (M^j)^{m_j}] &= \prod_{j=1}^k \E[(\sqrt{j}Z_j+\eta_j)^{m_j}],\qquad\qquad  & & \quad \sum_{j=1}^k jm_j\leq 2n-1 \\
\E_{O(2n)^-}[\prod_{j=1}^k \tr (M^j)^{m_j}] &= \prod_{j=1}^k \E[(\sqrt{j}Z_j-\eta_j)^{m_j}],  & & \quad \sum_{j=1}^k jm_j\leq 2n-1 \\
\E_{O(2n+1)^+}[\prod_{j=1}^k \tr (M^j)^{m_j}] &= \prod_{j=1}^k \E[(\sqrt{j}Z_j-(1-\eta_j))^{m_j}],  & & \quad \sum_{j=1}^k jm_j\leq 2n \\
\E_{O(2n+1)^-}[\prod_{j=1}^k \tr (M^j)^{m_j}] &= \prod_{j=1}^k \E[(\sqrt{j}Z_j+(1-\eta_j))^{m_j}],  & & \quad \sum_{j=1}^k jm_j\leq 2n \\
\E_{Sp(2n)}[\prod_{j=1}^k \tr (M^j)^{m_j}] &= \prod_{j=1}^k \E[(\sqrt{j}Z_j-\eta_j)^{m_j}], & & \quad \sum_{j=1}^k jm_j\leq 2n+1
\end{alignat*}
where the $Z_j$ are independent standard normal variables and $\eta_j=\frac{1+(-1)^j}{2}$.
\end{prop}

\begin{remark}\label{meanoftrace}
As we noted in the introduction, the joint eigenvalue densities do not take into account the deterministic eigenvalues of $O(2n)^-$ and $O(2n+1)^\pm$. Raising these to the power of $j$ and adding them to $\tr  M^j$ above shows that $\E_{G(n)} [\tr  M^j]$, if including both random and deterministic eigenvalues, is actually $\eta_j$ for all orthogonal matrices.
\end{remark}

\begin{remark}
These moments were first computed by Diaconis and Shahshahani in \cite{DS} in the case where $U$ belongs to the unitary, orthogonal or symplectic group (for half the range in the last two cases, i.e. for $\sum_{j=1}^k jm_j< n/2$ where $n$ is the size of the matrix) using the representation theory of these groups. Later Hughes and Rudnick obtained the identities for $SO(n)$ and $Sp(n)$ by a combinatorial argument involving the cumulants of linear eigenvalue statistics, see \cite{HR}. The unitary case was treated in \cite{J} (appendix) by using an identity for Toeplitz determinants, a method of proof which is similar to our proof below. See also \cite{PV} and \cite{STOLZ} for the moments of $U(n)$, $O(n)$, $SO(n)$ and $Sp(n)$.
\end{remark}

\begin{proof}
A partition $\lambda$ is a sequence of non-negative integers $\lambda_1\geq\lambda_2\geq\lambda_3\geq\dots$ which are called the parts of the partition. The sum of all the parts is finite and gives the weight $\lvert\lambda\rvert$. We will write $\lambda\vdash n$ to say that $\lambda$ is a partition of $n$, i.e. $\lvert\lambda\rvert=n$. The number of parts of $\lambda$ equal to $i$ is called the multiplicity of $i$ in $\lambda$ and is denoted by $m_i$, so $\lambda = (1^{m_1}2^{m_2}\dots k^{m_k})$. 
We will use the identity
\begin{equation}
\label{expformula}
\mathrm{exp}\Big(\sum_{n=1}^\infty \frac{g(n)}{n}t^n\Big) = \sum_{\lambda} t^{\lvert\lambda\rvert}z_\lambda^{-1}g(\lambda)
\end{equation}
where the sum is over all partitions. Here $g$ is a complex function on $\N$, $g(\lambda)=g(\lambda_1)g(\lambda_2)\dots$ with $g(0)=1$, and $z_\lambda = \prod_{i\geq1} m_i!i^{m_i}$ where $m_i$ is the multiplicity of $i$ in $\lambda$. This is a case of the "Exponential formula" (see 5.1.9 in \cite{Sta}). Here is a direct proof: it follows from the multinomial theorem that
$$ \big(\sum_{n=1}^\infty x_n t^n\Big)^k = \sum_{\substack{m_1+m_2+\dots=k,\\ m_i\in\N}} k!\prod_{i\geq1} \frac{(x_it^i)^{m_i}}{m_i!} $$
Thus,
\begin{align*}
\mathrm{exp}\Big(\sum_{n=1}^\infty \frac{g(n)}{n}t^n \Big) &= \sum_{k=0}^\infty \sum_{\substack{m_1+m_2+\dots=k,\\ m_i\in\N}} \prod_{i\geq1} \frac{g(i)^{m_i} t^{im_i}}{m_i!i^{m_i}} \\
&= \sum_{n=0}^\infty t^n \sum_{\substack{\sum_i im_i=n,\\ m_i\in\N}} \prod_{i\geq1} \frac{g(i)^{m_i}}{m_i!i^{m_i}}  \\
&= \sum_{n=0}^\infty \sum_{\lambda\vdash n} \frac{t^{\lvert\lambda\rvert}}{z_\lambda} \prod_{i\geq1} g(\lambda_i).
\end{align*}

We will now prove the proposition for $O(2n^+)$, the other cases being similar. Define
$$ f(e^{i\theta}) \coloneqq \prod_{i=1}^m \frac{1}{1-a_ite^{i\theta}}\frac{1}{1-a_ite^{-i\theta}} = \mathrm{exp}\Big(\sum_{k=1}^\infty \frac{p_k(a)t^k}{k}(e^{ik\theta}+e^{-ik\theta}) \Big), \qquad \theta\in[0,\pi), $$
where $a_i,t\in\D$, and $p_k(a)= \sum_{i=1}^m a_i^k$. Then
\begin{align*}
\prod_{j=1}^{n} f(e^{i\theta_j}) = \mathrm{exp}\Big(\sum_{k=1}^\infty \frac{p_k(a)p_k(e^{i\theta},e^{-i\theta})}{k}t^k \Big)
\end{align*}
where $e^{i\theta}=(e^{i\theta_1},\dots,e^{i\theta_n})$, $e^{-i\theta}=(e^{-i\theta_1},\dots,e^{-i\theta_n})$, and $p_k(e^{i\theta},e^{-i\theta})=\sum_{j=1}^{n} e^{ik\theta_j}+e^{-ik\theta_j}$. So by (\ref{expformula}),
$$ \prod_{j=1}^{n} f(e^{i\theta_j}) = \sum_{\lambda} t^{\lvert\lambda\rvert}z_\lambda^{-1}p_\lambda(a)p_\lambda(e^{i\theta},e^{-i\theta}) $$
which gives
\begin{equation}\label{expect1}
\E_{O(2n)^+}\Big[\prod_{j=1}^{n} f(e^{i\theta_j})\Big] = \sum_{\lambda} t^{\lvert\lambda\rvert}z_\lambda^{-1}p_\lambda(a)\E_{O(2n)^+}[p_\lambda(e^{i\theta},e^{-i\theta})].
\end{equation}
Observe that if $\lambda = (1^{m_1},2^{m_2},\dots k^{m_k})$, then
$$ \E_{O(2n)^+}[p_\lambda(e^{i\theta},e^{-i\theta})] = \E_{O(2n)^+}[\prod_{j=1}^k \tr (M^j)^{m_j}].$$
On the other hand, since $f(e^{i\theta})=f(e^{-i\theta})$, Proposition \ref{BE} applies. Indeed, if we define
$$ h(x) = \mathrm{exp}\Big(\sum_{k=1}^\infty \frac{2p_k(a)t^k}{k}T_k(x) \Big), \qquad x\in [-1,1],  $$
where $T_k$ is the $k$th Chebyshev polynomial, we see that
$$\E_{O(2n)^+}\Big[\prod_{j=1}^{n} f(e^{i\theta_j})\Big] = \E_{n,m}^{--} \Big[ \prod_{j=1}^{n} h(x_j) \Big], \qquad x_j= \cos \theta_j . $$
to which we can apply Lemma \ref{Kurt} and Proposition \ref{BE} ($\log(h\circ\cos)_+ = \log(f)_+ \in B_{1+}^1$ because $a_i,t\in\D$) and obtain
\begin{equation}\label{v2}
\E_{O(2n)^+}\Big[\prod_{j=1}^{n} f(e^{i\theta_j})\Big] = \mathrm{exp}\Big(\sum_{n=1}^\infty [\log g]_{2n}+\frac{1}{2}\sum_{n=1}^\infty n[\log g]_n^2 \Big) \det(1+Q_nH(e^{i\theta}g_+^{-1}\tilde{g_+})Q_n)
\end{equation}
where $g:[0,\pi)\to \C$, $g= h\circ \cos$.
We would like the Fredholm determinant to be 1, which happens if $m$ is sufficiently small since then $H(e^{i\theta}g_+^{-1}\tilde{g_+})$ is annihilated by $Q_n$. We have that
$$ g_+^{-1}(\theta)\tilde{g_+}(\theta) = \mathrm{exp}\Big(\sum_{k=1}^\infty \frac{p_k(a)t^k}{k}(-e^{ik\theta}+e^{-ik\theta}) \Big) = \prod_{i=1}^m \frac{1-a_ite^{i\theta}}{1-a_ite^{-i\theta}} $$
which can be expanded into
$$
g_+^{-1}(\theta)\tilde{g_+}(\theta) = \sum_{0\leq j\leq m} (-te^{i\theta})^je_j(a) \sum_{0\leq k} (te^{-i\theta})^kh_k(a) = \sum_{0\leq j\leq m} \sum_{k\leq j} (-1)^jt^{2j-k}e_j(a)h_{j-k}(a)e^{ik\theta}
$$
where $e_k$ is the $k$th elementary symmetric polynomial and $h_k$ is the $k$th complete symmetric polynomial. Hence the $(x,y)$ element of $H(e^{i\theta}g_+^{-1}\tilde{g_+})$ is different from zero if and only if $x+y\leq m$ (recall that $H(c) = (c_{j+k+1})_{j,k=0}^\infty$) so $Q_nH(e^{i\theta}g_+^{-1}\tilde{g_+})Q_n$ is zero for $2n>m$. In this case (\ref{v2}) becomes
$$ \E_{O(2n)^+}\Big[\prod_{j=1}^{n} f(e^{i\theta_j})\Big] = \mathrm{exp} \Big( \sum_{n=1}^\infty \frac{
p_{2n}(a)}{2n}t^{2n} +\frac{1}{2}\sum_{n=1}^\infty \frac{
p_{n}(a)^2}{n}t^{2n} \Big). $$
Now, by definition of $\eta_n$ and because $\E[\mathrm{exp(\xi Z_n)}]=\mathrm{exp}(\xi^2/2),$
\begin{align*}
\E_{O(2n)^+}\Big[\prod_{j=1}^{n} f(e^{i\theta_j})\Big] &= \mathrm{exp} \Big( \sum_{n=1}^\infty  \frac{
p_{n}(a)t^{n}}{n}\eta_n +\frac{n}{2}\Big(\frac{
p_{n}(a)t^{n}}{n}\Big)^2 \Big) \\
&=  \E\Big[ \mathrm{exp} \Big( \sum_{n=1}^\infty  \frac{
p_{n}(a)t^{n}}{n}(\eta_n +\sqrt{n}Z_n) \Big)\Big].
\end{align*}
Using (\ref{expformula}) one more time and applying the dominated convergence theorem gives
$$ \E_{O(2n)^+}\Big[\prod_{j=1}^{n} f(e^{i\theta_j})\Big] = \sum_{\lambda} t^{\lvert\lambda\rvert}z_\lambda^{-1}p_\lambda(a)\E\Big[\prod_{i\geq1}(\lambda_iZ_{\lambda_i}+\eta_{\lambda_i})\Big].$$
By comparing with (\ref{expect1}) and matching coefficients of $t^j$, we see that
$$ \sum_{\lambda\vdash j} z_\lambda^{-1}p_\lambda(a)\E_{O(2n)^+}[p_\lambda(e^{i\theta},e^{-i\theta})] = \sum_{\lambda\vdash j} z_\lambda^{-1}p_\lambda(a)\E\Big[\prod_{i\geq1}(\lambda_iZ_{\lambda_i}+\eta_{\lambda_i})\Big], \qquad j\geq 1. $$
The power sums $p_\lambda(a)$, $\lambda\vdash j$, form a basis of the ring of symmetric polynomials of degree $j$ in $j$ variables. Therefore, if $\lvert\lambda\rvert = m \leq 2n -1$ i.e. if $\sum_{j=1}^k jm_j \leq 2n-1 $,
$$ \E_{O(2n)^+}[\prod_{j=1}^k \tr (M^j)^{m_j}] = \E_{O(2n)^+}[p_\lambda(e^{i\theta},e^{-i\theta})] = \E\Big[\prod_{i\geq1}(\lambda_iZ_{\lambda_i}+\eta_{\lambda{i}})\Big]. $$
\end{proof}

Thus, by the Cram\'{e}r-Wold device followed by the method of moments for normal distributions, we see that the random vector $\mathbf{X}=(X_1, X_2, \dots X_m)$ with $X_j=\frac{1}{\sqrt{j}}(\tr U^j-\E_{G(n)} [ \tr U^j])$ and $m$ fixed converges in distribution to $(Z_1,Z_2, \dots Z_m)$ where the $Z_j$ are independent standard normal variables. Also, $\E_{G(n)} [ \tr U^j] = \pm \eta_j$ or $\E_{G(n)} [ \tr U^j] = \pm (1-\eta_j)$, whether the size of the matrix is even or odd. This can also be seen by considering the characteristic function of  $\mathbf{X}$, as we do in the next corollary.
First recall that
\begin{align*}
f(x)=\sum_{1\leq k\leq m} \frac{\xi_k}{\sqrt{k}}\Big(2\mathrm{T_k}(x)-\frac{\E_{G(n)} [ \tr U^k]}{n}\Big) , \quad x\in [-1,1]    
\end{align*}
and that
\begin{equation*}
g(\theta)= f(\cos(\theta)) = \sum_{1\leq k\leq m} \frac{\xi_k}{\sqrt{k}} \Big(2\cos{k\theta}-\frac{\E_{G(n)} [ \tr U^k]}{n}\Big),\quad \theta\in[0,\pi).    
\end{equation*}

\begin{cor}
\label{Fredholm}
Let $F_{n,m}^{a,b}$ be the characteristic function of $\mathbf{X}=(X_1, X_2, \dots X_m)$, $X_j=\frac{1}{\sqrt{j}}(\tr U^j-\E_{G(n)} [ \tr U^j])$ and $U$ a Haar distributed orthogonal or symplectic matrix. Then, for $\xi\in\R^m$,
\begin{align*}
F_{n,m}^{-+}(\xi) &= e^{-\frac{1}{2}\|\xi\|^2}\det(1+H(e^{-i2n\theta}e^{2\Im{g_+}})) \\
F_{n,m}^{+-}(\xi) &= e^{-\frac{1}{2}\|\xi\|^2}\det(1-H(e^{-i2n\theta}e^{2\Im{g_+}})) \\
F_{n,m}^{++}(\xi) &= e^{-\frac{1}{2}\|\xi\|^2}\det(1-H(e^{-i(2n+1)\theta}e^{2\Im{g_+}})) \\
F_{n,m}^{--}(\xi) &= e^{-\frac{1}{2}\|\xi\|^2}\det(1+H(e^{-i(2n-1)\theta}e^{2\Im{g_+}}))
\end{align*}
where $H$ denotes the Hankel matrix i.e. $H(f) = (f_{j+k+1})_{j,k=0}^\infty$, and $f_k$ is the $k$th Fourier coefficient of $f$.
\end{cor}

\begin{proof}
Again, consider the first case for simplicity; the proof for the remaining ones is essentially the same. By definition of the characteristic function and Lemma \ref{Kurt}
$$ F_{n,m}^{-+}(\xi) = \E_n^{-+}\Big[ \prod_{1\leq j\leq n} e^{if(x_j)} \Big] = \det ((\widehat{e^{ig}})_{j-k}+(\widehat{e^{ig}})_{j+k+1})_{0\leq i,j\leq n-1}. $$
The function $g$ satisfies $\tilde{g}=g$ and, being a trigonometric polynomial, belongs to $B_1^1$. Thus the conditions of Proposition \ref{BE} are met and we obtain
\begin{align*}
F_{n,m}^{-+}(\xi) &= \exp\Big(-n\sum_{k=0}^m\frac{\xi_k\E_{O(2n+1)^-}[\tr U^k]}{\sqrt{k}n}+\sum_{k=0}^{\floor{(m-1)/2}}\frac{\xi_{2k+1}}{\sqrt{2k+1}}-\frac{1}{2}\sum_{k=1}^m \xi_k^2\Big) \det(1+H(e^{-i2n\theta}e^{2\Im{g_+}})) \\
&= e^{-\frac{1}{2}\sum_{k=1}^m \xi_k^2}\det(1+H(e^{-i2n\theta}e^{2\Im{g_+}}))
\end{align*}
where we used that $\E_{O(2n+1)^-} [ \tr U^j] = 1-\eta_j$.
\end{proof}

The Fredholm determinants in Corollary \ref{Fredholm} converge to one as $n$ tends to infinity. This can be seen for example from
\begin{equation}\label{bound_det}
|1-\det(1+K)| \leq e^{\| K\|_{J_1}}-1 \leq \|K\|_{J_1}e^{\| K\|_{J_1}}
\end{equation}
(Theorem 3.3 in \cite{GGK}), where $K$ is any trace-class operator and $\mathrm{J}_1$ is the trace norm, and the fact that $\lim_{n\to0} \|Q_nKQ_n\|_{\mathrm{J}_1} = 0$. This in turn follows from $\lim_{n\to0} \|K-(1-Q_n)K(1-Q_n)\|_{\mathrm{J}_1} = 0$, (Proposition 4.2 in \cite{GGK}) and 
$$ Q_nKQ_n = Q_n(K-(1-Q_n)K(1-Q_n))Q_n$$
which implies
\begin{equation}\label{Qn}
\lim_{n\to0} \|Q_nKQ_n\|_{\mathrm{J}_1} \leq \lim_{n\to0} \|K-(1-Q_n)K(1-Q_n)\|_{\mathrm{J}_1} =0    
\end{equation}
where we used that $\|Q_n\|=1$ and the inequality (Proposition 4.2 in \cite{GGK})
\begin{equation}\label{abc}
\|ABC\|_{J_1} \leq \|A\| \|B\|_{J_1} \|C\|.    
\end{equation}
The fact that our Hankel operators are trace class is part of Proposition \ref{BE} in \cite{BE}. In our case it also follows from the identity 
$$ H(ab) = T(a)H(b)+ H(a)T(\tilde{b})$$
which holds for $a,b\in L_\infty(\T)$, and which gives
$$ H(a_+^{-1}\tilde{a_+}) = T(\tilde{a_+})H(a_+^{-1})$$
since $H(\tilde{a_+})=0$. Now, $a_+^{-1}= e^{-ig_+}$ and $g\in B_1^1$, so $a_+^{-1}\in B_1^1$ which is equivalent to $H(a_+^{-1})$ and $H(\tilde{a}_+^{-1})$ being trace class (see \cite{P}). Applying (\ref{abc}) to $A=T(\tilde{a_+})$, $B=H(a_+^{-1})$, $C=I$, and using that the operator norm of a Toeplitz operator is given by the $L_\infty$-norm of its symbol proves the statement.

Hence the characteristic function $F_{n,m}^{a,b}$ converges to that of a standard normal vector, and the speed of this convergence depends on how fast the Fredholm determinants approach one. To measure the rate of convergence we first consider the $L_2$ distance between the probability density of $\mathbf{X}$ and that of a normal random vector, which by Plancherel's theorem amounts to study the $L_2$ distance between their characteristic functions. To this end we divide the real line into three parts, or regimes, where we use different techniques to obtain as good estimates as possible. We start with the first regime which is treated using Corollary \ref{Fredholm} and which will give us the dominant error term.

\section{Gaussian approximation}

Our goal is to prove

\begin{prop}\label{regime1}
Let $N=n/m$ and set
\begin{equation}\label{lambda1}
\Lambda_1=\frac{n}{2m\sqrt{\log m+1}}. 
\end{equation}
If $N\geq m$ then, for any pair $(a,b)=(\pm1/2,\pm1/2)$,
\begin{multline}
\Big(\int_{\|\xi\|<\Lambda_1}  \lvert F_{n,m}^{a,b}(\xi)- e^{-\|\xi\|^2/2}\rvert^2 d\xi \Big)^{1/2}\\  \leq \frac{16}{15}e^{13/24}(e^{9/8}+1)  \frac{m^{3/2}\sqrt{\Omega_m}}{\sqrt{N}} \Big(\frac{m}{2}\Big)^{\frac{m}{4}}\frac{(e^{3/2}(\log m +1))^{N}}{\sqrt{\Gamma(2N+1)}}
\end{multline}
where $\Omega_m$ is the volume of the unit $m$-ball.
\end{prop}

We start by giving two lemmas that we will need later. The first one is essentially Lemma 4.1 in \cite{JL}; we include the proof here for completeness.

\begin{lemma}\label{bound_fourier}
Fix $m\in\N$ and $\xi\in\R^m$. Let $\rho = \sqrt{\log m+1}\|\xi\|$. If $k\geq2m\rho$, then
$$ \lvert(\widehat{e^{2\Im g_+}})_k \rvert \leq 2e^{\rho}\frac{\rho^{\ceil{k/m}}}{\ceil{k/m}!}.$$
\end{lemma}

\begin{proof}
Let $M\geq1$ and define
$$\phi_M(\omega) = \sum_{k=0}^M \frac{w^k}{k!}, \quad \omega\in\C. $$
Then, if $Mm<k$,
\begin{equation*}
\int_{-\pi}^\pi \phi_M(-ig_+(\theta))e^{i\bar{g_+}(\theta)-ik\theta}\frac{d\theta}{2\pi} = 0 
\end{equation*}
which gives
\begin{multline*}
\lvert (\widehat{e^{2\Im g_+}})_k \rvert = \Big\lvert \int_{-\pi}^\pi e^{-i(g_+(\theta)-\bar{g_+}(\theta))-ik\theta}\frac{d\theta}{2\pi} \Big\rvert = \Big\lvert \int_{-\pi}^\pi (e^{-ig_+(\theta)} - \phi_M(-ig_+(\theta)))e^{i\bar{g_+}(\theta)-ik\theta} \frac{d\theta}{2\pi} \Big\rvert \\ \leq \int_{-\pi}^\pi \lvert e^{-ig_+(\theta)}-\phi_M(-ig_+(\theta)) \rvert e^{\Im g_+(\theta)} \frac{d\theta}{2\pi}. 
\end{multline*}
The inequality $(k+j)!\geq k!(k+1)^j$ for $k,j\in\N$ gives
$$ \lvert e^\omega-\phi_M(\omega) \rvert = \lvert\sum_{k\geq M+1} \frac{\omega^k}{k!} \rvert \leq \frac{\lvert\omega\rvert^{M+1}}{(M+1)!}\sum_{k\geq0} \Big(\frac{\lvert\omega\rvert}{M+2}\Big)^k \leq 2\frac{\lvert\omega\rvert^{M+1}}{(M+1)!}  $$
if $\lvert\omega\rvert \leq M/2 +1$. Combined with
$$ \|g_+\|_\infty \leq \sum_{k=1}^\infty \frac{\lvert\xi_k\rvert}{\sqrt{k}} \leq \sqrt{1+\log m}\|\xi\| = \rho $$
which follows from the Cauchy-Schwarz inequality, we obtain
$$ \lvert (\widehat{e^{2\Im g_+}})_k \rvert \leq 2e^\rho\frac{\rho^{M+1}}{(M+1)!} $$
for $k>Mm$ and $M+2\geq 2\rho$. These inequalities are satisfied if $M+1= \ceil{k/m}$ and $k\geq 2m\rho$, which proves the statement.
\end{proof}

\begin{lemma}\label{tech}
Let $y>0$. The function $h(x) = \frac{y^x}{\Gamma(x+1)}$ is decreasing on $[y,\infty)$.
\end{lemma}
\begin{proof}
The derivative of $h$ is
$$h'(x) = \frac{y^x(\log y -\psi(x+1))}{\Gamma(x+1)},$$
where $\psi$ is the Digamma function, i.e the logarithmic derivative of $\Gamma(x+1)$, which has the following integral representation (5.9.12 \cite{DLMF})
$$ \psi(x+1) = \int_0^\infty \frac{e^{-t}}{t}-\frac{e^{-tx}}{e^t-1} dt. $$
Now observe that
\begin{multline*}
\int_0^\infty \frac{e^{-t}-e^{-tx}}{t}dt = \lim_{ R\rightarrow\infty}\lim_{\epsilon\rightarrow 0 } \Big( \int_\epsilon^R \frac{e^{-t}}{t} dt - \int_{\epsilon x}^{Rx} \frac{e^{-t}}{t} dt\Big) 
= \lim_{ R\rightarrow\infty}\lim_{\epsilon\rightarrow 0 } \Big( \int_\epsilon^{\epsilon x} \frac{e^{-t}}{t}dt - \int_{R}^{Rx} \frac{e^{-t}}{t} dt \Big) \\
= \lim_{ R\rightarrow\infty}\lim_{\epsilon\rightarrow 0 } \Big( \log x + \int_\epsilon^{\epsilon x} \frac{e^{-t}-1}{t}dt - \int_{R}^{Rx} \frac{e^{-t}}{t} dt \Big) = \log x.
\end{multline*}
Thus
$$ \psi(x+1) = \int_0^\infty \frac{e^{-t}}{t}-\frac{e^{-tx}}{e^t-1} dt > \log x $$
since $e^t-1> t$. This shows that $h'(x)<0$ on $[y,\infty)$, assuming $y>0$.
\end{proof}

Proposition \ref{regime1} is based on the following bound on the distance between the characteristic function $F_{n,m}^{a,b}$ of our random vector $\mathbf{X}$ and that of a standard Gaussian vector. 
\begin{prop}\label{regime1a}
Let $N=n/m$, $\rho = \sqrt{\log m +1}\|\xi\|$. If $N\geq 2\rho \vee m$ then
\begin{align}
\lvert F_{n,m}^{a,b}(\xi)- e^{-\|\xi\|^2/2}\rvert \leq \frac{32 e^{1/2}(e^{9/8}+1)}{15}\frac{me^\rho \rho^{2N}e^{-\|\xi\|^2/2}}{\Gamma(2N+1)}.
\end{align}
\end{prop}

\begin{proof}
The proof is based on the following inequality
\begin{equation} \label{eq1}
\lvert 1-\det(1+K)\rvert \leq e^{\Upsilon_{n,m}(\xi)}(e^{(\Upsilon_{n,m}(\xi)+1)^2/2}+1)\Upsilon_{n,m}(\xi)
\end{equation}
which holds for any trace-class operator $K$ and where $\Upsilon_{n,m}(\xi)=\max\{\tr K,\|K\|_{\mathcal{J}_2}\}$, $\| \cdot \|_{\mathcal{J}_2}$ being the Hilbert-Schmidt norm. It is obtained as follows: the generalised determinant $\det_2$ can be defined as $\det_2(I+K) = e^{-\tr (K)}\det(I+H)$. Hence, by definition,
\begin{multline*}
  |1-\det(1+K)| = |1-e^{\tr  K}\det_2(1+K)| \leq |e^{\tr  K}||1-\det_2(1+K)|+|1-e^{\tr  K}| \\ \leq |e^{\tr  K}||1-\det_2(1+K)|+|\tr  K\rvert \lvert e^{\tr  K}|.  
\end{multline*}
Now we use that
$$ |1-\det_2(1+K)| \leq \|K\|_{J_2}e^{\frac{1}{2}(1+\|K\|_{J_2})^2}$$
which is part of Theorem 9.2. in \cite{Simon}. Observe that we could use (\ref{bound_det}) instead of (\ref{eq1}) but the latter gives us a slightly better bound.

We will apply (\ref{eq1}) to the Fredholm determinants in Corollary (\ref{Fredholm}). First consider the case $(a,b)=(-1/2,-1/2)$, i.e. $K=Q_nH(e^{i\theta}e^{2\Im{g_+}})Q_n$.
By Lemma \ref{bound_fourier}, if $n\geq m\rho$,
\begin{equation}\label{J2a}
\|Q_nH(e^{i\theta}e^{2\Im{g_+}})Q_n\|_{\mathcal{J}_2}^2 = \sum_{k\geq 2n} (k-2n+1) \lvert(\widehat{e^{2\Im g_+}})_k \rvert^2 \leq 4 e^{2\rho} \sum_{k\geq 0} (k+1) \frac{\rho^{2\ceil{\frac{k}{m}+2N}}}{\ceil{\frac{k}{m}+2N}!^2}.
\end{equation}
If we let $k=jm+r$, $0\leq r < m$, then
\[\sum_{k\geq 0} (k+1) \frac{\rho^{2\ceil{\frac{k}{m}+2N}}}{\ceil{\frac{k}{m}+2N}!^2} = \sum_{j=0}^\infty \sum_{r=0}^{m-1} (jm+r+1) \frac{\rho^{2\ceil{j +\frac{r}{m}+2N}}}{\ceil{j +\frac{r}{m}+2N}!^2} \leq m \sum_{j=0}^\infty \sum_{r=0}^{m-1} (j+1) \frac{\rho^{2\ceil{j +\frac{r}{m}+2N}}}{\ceil{j +\frac{r}{m}+2N}!^2}.\]
Now, $\ceil{j +\frac{r}{m}+2N}\geq 2N \geq \rho$, so by Lemma \ref{tech}
\[\frac{\rho^{2\ceil{j +\frac{r}{m}+2N}}}{\ceil{j +\frac{r}{m}+2N}!^2} \leq \frac{\rho^{2(j +2N)}}{\Gamma(j +2N+1)^2}, \]
which combined with the inequality $\Gamma(j+x+1)\geq(x+1)^j\Gamma(x+1)$, $j\in\N$ (which follows from the recurrence relation $\Gamma(x+1)=x\Gamma(x)$) gives
\begin{align*}
\|Q_nH(e^{i\theta}e^{2\Im{g_+}})Q_n\|_{\mathcal{J}_2}^2 &\leq  4m e^{2\rho} \sum_{j=0}^\infty \sum_{r=0}^{m-1} (j+1)\frac{\rho^{2(j+2N)}}{(2N+1)^{2j}\Gamma(2N+1)^2} \\
&= \frac{4m^2 e^{2\rho}\rho^{4N}}{\Gamma(2N+1)^2} \sum_{j=0}^\infty (j+1)\Big(\frac{\rho}{2N+1}\Big)^{2j}.
\end{align*}
We set $c_* = \rho/(2N)$ and obtain
\begin{align}\label{J2b}
\|Q_nH(e^{i\theta}e^{2\Im{g_+}})Q_n\|_{\mathcal{J}_2}^2 \leq \frac{4m^2 e^{2\rho}}{\Gamma(2N+1)^2}  \frac{\rho^{4N}}{(1-c_*^2)^2}.  
\end{align}
Similarly, by Lemma \ref{bound_fourier}, if $n\geq m\rho$,
$$ \lvert \tr Q_nH(e^{i\theta}e^{2\Im{g_+}})Q_n \rvert \leq \sum_{k\geq n} \lvert(\widehat{e^{2\Im g_+}})_{2k} \rvert \leq 2e^{\rho}\sum_{k\geq 0} \frac{\rho^{\ceil{2k/m+2N}}}{\ceil{2k/m+2N}!}. $$
By Lemma \ref{tech} and because $0<\rho<\ceil{2N}$
\begin{align}\label{Tr}
\lvert \tr Q_nH(e^{i\theta}e^{2\Im{g_+}})Q_n \rvert &\leq 2e^{\rho}\sum_{k\geq 0} \frac{\rho^{2k/m+2N}}{\Gamma(2k/m+2N+1)} = 2e^{\rho}\sum_{j\geq 0}\sum_{r=0}^{m-1} \frac{\rho^{2j+2r/m+2N}}{\Gamma(2j +2r/m+2N+1)} \nonumber \\
&\leq 2me^{\rho}\sum_{j\geq 0} \frac{\rho^{2j+2N}}{\Gamma(2j+2N+1)} \leq \frac{2m e^{\rho}}{1-c_*^2}\frac{\rho^{2N}}{\Gamma(2N+1)}
\end{align}
where for the last inequality we used that $\Gamma(j+x+1)\geq(x+1)^j\Gamma(x+1)$. 
Hence
$$ \Upsilon_{n,m}(\xi) \leq  \frac{2m e^{\rho}}{1-c_*^2}\frac{\rho^{2N}}{\Gamma(2N+1)}. $$
A well-known estimate for the Gamma function is 
\begin{equation}\label{Stirling1}
\sqrt{2\pi}x^{x-1/2}e^{-x} < \Gamma(x) < \sqrt{2\pi}e^{1/(12x)}x^{x-1/2}e^{-x}, \qquad x>0,    
\end{equation}
which can be found in \cite{DLMF} (Equation 5.6.1). Equivalently, by multiplying the above inequalities with $x$ and using the functional equation $\Gamma(x+1)=x\Gamma(x)$, we obtain
\begin{equation}\label{Stirling}
\sqrt{2\pi}x^{x+1/2}e^{-x} < \Gamma(x+1) < \sqrt{2\pi}e^{1/(12x)}x^{x+1/2}e^{-x}, \qquad x>0, 
\end{equation}
and therefore
\begin{align*}
\Upsilon_{n,m}(\xi) \leq \frac{me^{\rho}}{\sqrt{\pi N}(1-c_*^2)}\Big(\frac{e\rho}{2N}\Big)^{2N} =\frac{m}{\sqrt{\pi N}(1-c_*^2)}(c_*e^{1+c_*})^{2N}.
\end{align*}
If we choose $c_*\leq 1/4$ (so that $c_*e^{1+c_*}<1$) and let $N\geq m$ we see that
\begin{align*}
 \Upsilon_{n,m}(\xi) \leq \frac{16}{15\sqrt{\pi}}\sqrt{N}\Big(\frac{e^{5/4}}{4}\Big)^{2N}.  
\end{align*}
The maximum is attained when $N=(\log 256 -5)^{-1}$, at which it is strictly less than $1/2$. Therefore, (\ref{eq1}) gives 
\begin{align} \label{integrand}
\lvert F_{n,m}^{--}(\xi)- e^{-\|\xi\|^2/2}\rvert \leq  e^{1/2}(e^{9/8}+1)\Upsilon_{n,m}(\xi) \leq \frac{32 e^{1/2}(e^{9/8}+1)}{15}\frac{me^\rho \rho^{2N}e^{-\|\xi\|^2/2}}{\Gamma(2N+1)}
\end{align}
if $N\geq 2\rho \vee m$. 

It is easy to extend this result to the other three cases, i.e. when $(a,b)=(1/2,1/2)$ or $(a,b)=\pm (1/2,-1/2)$. The operator $K$ in the Fredholm determinants appearing in Corollary (\ref{Fredholm}) is then equal to either $K= -Q_nH(e^{-i\theta}e^{2\Im g_+})Q_n$ or $K= \pm Q_nH(e^{2\Im g_+})Q_n$. Consider this latter case. By Lemma \ref{bound_fourier}, if $n\geq m\rho-1/2$, and in particular if $n\geq m\rho$,
$$\|\pm Q_nH(e^{2\Im{g_+}})Q_n\|_{\mathcal{J}_2}^2 = \sum_{k\geq 2n} (k-2n+1) \lvert(\widehat{e^{2\Im g_+}})_{k+1} \rvert^2 \leq 4 e^{2\rho} \sum_{k\geq 0} (k+1) \frac{\rho^{2\ceil{k/m+2N+1/m}}}{\ceil{k/m+2N+1/m}!^2}. $$
But by Lemma \ref{tech} the right-hand side is less than the upper bound in (\ref{J2a}). Thus (\ref{J2b}) still applies, and a similar argument shows that so does (\ref{Tr}), which all together lead to (\ref{integrand}). The case $K= -Q_nH(e^{-i\theta}e^{2\Im g_+})Q_n$ is treated similarly.
\end{proof}

\begin{proof}[Proof of Proposition \ref{regime1}]
By Proposition \ref{regime1a} we have
\begin{multline*}
\Big(\int_{\|\xi\|<\Lambda_1}  \lvert F_{n,m}^{a,b}(\xi)- e^{-\|\xi\|^2/2}\rvert^2 d\xi \Big)^{1/2} \leq \frac{32e^{1/2}(e^{9/8}+1)}{15} \frac{m e^{N/2}}{\Gamma(2N+1)} \Big(\int_{\|\xi\|<\Lambda_1} \rho^{4N}e^{-\|\xi\|^2} d\xi\Big)^{1/2} \\
\leq \frac{32e^{1/2}(e^{9/8}+1)}{15} \frac{m e^{N/2}}{\Gamma(2N+1)}(\log m +1)^{N} \Big(\int_{\R^m} \|\xi\|^{4N}e^{-\|\xi\|^2} d\xi\Big)^{1/2}.    
\end{multline*}
A change of variables to spherical coordinates shows that
\begin{align*}
\int_{\R^m} \|\xi\|^{4N}e^{-\|\xi\|^2} d\xi = \frac{m\Omega_m}{2}\Gamma(2N+m/2)
\end{align*}
where $\Omega_m$ is the volume of the unit $m$-ball. Stirling's approximation (Inequalities (\ref{Stirling1}) and (\ref{Stirling})) gives
$$ \sqrt{\frac{\Gamma(2N+m/2)}{\Gamma(2N+1)}} \leq \Big(\frac{ \sqrt{2\pi}e^{1/12}(2N+m/2)^{2N+\frac{m-1}{2}}e^{-2N-\frac{m}{2}}}{ \sqrt{2\pi} (2N)^{2N+\frac{1}{2}}e^{-2N}}\Big)^{1/2} \leq e^{1/24}\frac{(m/2)^{\frac{m}{4}}e^N}{\sqrt{2N}} $$
where we used the inequality $1+x\leq \mathrm{e}^x$, $x\in\R$ twice. This finishes the proof of the proposition.
\end{proof}

\section{Intermediate regime}
We will now estimate $F_{n,m}^{a,b}(\xi)$ directly from its integral expression, with probability density given by (\ref{density2}). Recall that 
\begin{equation*}
g(\theta)= f(\cos(\theta)) = \sum_{1\leq k\leq m} \frac{\xi_k}{\sqrt{k}} \Big(2\cos{k\theta}-\frac{\E_{G(n)} [ \tr U^k]}{n}\Big),\quad \theta\in[0,\pi).    
\end{equation*}

We will need the following inequality that appears in \cite{JL}.

\begin{lemma}\label{ineq}
If $y\in[-1,1]$, $y\neq 0$ and $x\in\R$, then
$$ 1+ \Big(\frac{\sinh(x)}{y}\Big)^2 \leq \exp\Big(\frac{x}{y}\Big)^2. $$
\end{lemma}
\begin{proof}
It suffices to prove the inequality for $x\geq0$. Set $\kappa(x)= (1+\sinh^2(x)/y^2)e^{-x^2/y^2}$ and differentiate:
\begin{align*}
\kappa'(x) &= \frac{2e^{-x^2/y^2}}{y^2}\Big(\sinh(x)\cosh(x)-x\Big(1+\frac{\sinh^2(x)}{y^2}\Big)\Big) \\
&\leq \frac{2e^{-x^2/y^2}}{y^2}(\sinh(x)\cosh(x)-x(1+\sinh^2(x))) \\
&= \frac{2e^{-x^2/y^2}}{y^2}\cosh(x)(\sinh(x)-x\cosh(x)) \leq 0.
\end{align*}
Thus $\kappa(x) \leq \kappa(0) = 1$ if $x\geq 0$ which proves the result for $x\in\R$.
\end{proof}

In the next proposition, only part (a) is needed for the intermediate regime but we include part (b) which is used in the last section and has a similar proof.

\begin{prop}\label{start}
\begin{enumerate}[(a)]
\item Let $\xi\in\R^m$ and $\nu\in\R$ be such that
\begin{align}\label{nu1}
0< 2\frac{\nu}{n}\sqrt{\log m +1}\|\xi\| < \frac{c_0}{m}
\end{align}
for any constant $c_0>0$. Let $h$ be the Hilbert transform of $g$, i.e. $h(\theta)=\sum_{k=1}^m 2\xi_k\sin{k\theta}/\sqrt{k}$. If $a=b=-1/2$, then 
\begin{multline*}
| F_{n,m}^{a,b}(\xi)| \leq \mathrm{exp}\Big(\Big(2-\frac{1}{n}\Big)\nu^2m(m+1)\|\xi\|^2+\frac{2e^{c_0}\nu^2}{n}(m+1)^2(\log m+1)\|\xi\|^3\Big) \\ \times \E_{O(2n)^+}[\prod_{1\leq j\leq n}\mathrm{exp}\Big(\frac{\nu}{n}h(\theta_j)g'(\theta_j)\Big)].
\end{multline*}
In the three other cases,
\begin{multline*}
| F_{n,m}^{a,b}(\xi)| \leq \mathrm{exp}\Big(\Big(2+\frac{1}{n}\Big)\nu^2m(m+1)\|\xi\|^2+\frac{2e^{c_0}\nu^2}{n}(m+1)^2(\log m+1)\|\xi\|^3\Big) \\ \times \E_{G(n)}[\prod_{1\leq j\leq n}\mathrm{exp}\Big(\frac{\nu}{n}h(\theta_j)g'(\theta_j)\Big)].
\end{multline*}

\item Let $\xi\in\R^m$ and set
\begin{align}\label{nubis}
\nu = \frac{\eta\sqrt{n}}{(m+1)^2\|\xi\|}, \qquad \lambda= \frac{\nu}{n}\Big(1-\frac{\nu^2}{n^2}\frac{e^{\sqrt{2/n}\eta}}{3\sqrt{3}}(m+1)^4\|\xi\|^2\Big)
\end{align}
for any constant $\eta\in(0,1]$. If $a=b=-1/2$, then 
\begin{equation*}
| F_{n,m}^{a,b}(\xi)| \leq e^{\eta^2(n-\frac{1}{2})} \E_{O(2n)^+}[\prod_{1\leq j\leq n}e^{-\lambda g'(\theta_j)^2}].
\end{equation*}
In the three other cases,
\begin{equation*}
| F_{n,m}^{a,b}(\xi)| \leq e^{\eta^2(n+\frac{1}{2})} \E_{G(n)}[\prod_{1\leq j\leq n}e^{-\lambda g'(\theta_j)^2}].    
\end{equation*}
\end{enumerate}
\end{prop}

\begin{proof}
We start with the case $a=b=-1/2$. Then 
\begin{align*}
 F_{n,m}^{a,b}(\xi) &= \E_{O(2n)^+}[e^{i\tr g(U)}] \\
 &= \frac{2^{(n-1)^2}}{n!\pi^n}\int_{[0,\pi]^n} \prod_{1\leq j\leq n} e^{ig(\theta_j)} \prod_{1\leq j<k\leq n}4\sin^2\Big(\frac{\theta_j-\theta_k}{2}\Big) \sin^2\Big(\frac{\theta_j+\theta_k}{2}\Big) \prod_{1\leq j\leq n} d\theta_j.    
\end{align*}
Let $\gamma$ be the curve in the complex plane given by $\gamma(t)=t-i\nu h(t)/n$, $t\in[0,\pi]$ where $\nu\in\R$ and $h$ is continuous and satisfies $h(0)=h(\pi)=0$. Since the integrand above has an analytic continuation in $\C^n$ we can deform the contour into the curve $\gamma\{[0,\pi]\}^n$ and then make a change of variables to get back the original contour. We obtain
\begin{align*}
F_{n,m}^{a,b}(\xi)&= \frac{2^{(n-1)^2}}{n!\pi^n}\int_{[0,\pi]^n} \prod_{1\leq j\leq n} e^{ig(\theta_j-i\frac{\nu}{n}h(\theta_j))}\prod_{1\leq j<k\leq n}4\sin^2\Big(\frac{\theta_j-\theta_k-i\frac{\nu}{n}(h(\theta_j)-h(\theta_k))}{2}\Big)\\
& \sin^2\Big(\frac{\theta_j+\theta_k-i\frac{\nu}{n}(h(\theta_j)+h(\theta_k))}{2}\Big) \prod_{1\leq j\leq n} \Big(1-i\frac{\nu}{n}h'(\theta_j)\Big)d\theta_j.
\end{align*}
Taking the absolute value gives the upper bound
\begin{align}\label{bound}
| F_{n,m}^{a,b}(\xi)| &\leq \frac{2^{(n-1)^2}}{n!\pi^n}\int_{[0,\pi]^n} \prod_{1\leq j\leq n} \lvert e^{ig(\theta_j-i\frac{\nu}{n}h(\theta_j))}\rvert \prod_{1\leq j<k\leq n}4\sin^2\Big(\frac{\theta_j-\theta_k}{2}\Big)\sin^2\Big(\frac{\theta_j+\theta_k}{2}\Big)\\
& e^{\frac{\nu^2}{n^2} (H(\theta_j,\theta_k)+H(\theta_j,-\theta_k))} \prod_{1\leq j\leq n} e^{\frac{\nu^2}{2n^2}h'(\theta_j)^2}d^n\theta \nonumber
\end{align}
where $$H(x,y)=\Big(\frac{h(x)-h(y)}{2\sin(\frac{x-y}{2})}\Big)^2.$$ 
Here we used that $|1+ia|\leq e^{a^2/2}$ and that 
\begin{align*}
\Big|\sin^2\Big(\frac{\theta_j\pm \theta_k}{2}-i\frac{\nu}{2n}(h(\theta_j)\pm h(\theta_k))\Big)\Big| &= \sin^2\Big(\frac{\theta_j\pm\theta_k}{2}\Big)+\sinh^2\Big(\frac{\nu}{2n}(h(\theta_j)\pm h(\theta_k))\Big) \\
 &= \sin^2\Big(\frac{\theta_j\pm\theta_k}{2}\Big)+\sinh^2\Big(\frac{\nu}{2n}(h(\theta_j)- h(\mp \theta_k))\Big) \\
&\leq \sin^2\Big(\frac{\theta_j\pm\theta_k}{2}\Big) e^{\frac{\nu^2}{n^2} H(\theta_j,\mp\theta_k)}     
\end{align*}
where the inequality follows from Lemma \ref{ineq}. We now specialize the proof to part (a), i.e. we assume that (\ref{nu1}) holds and that $h(\theta)=\sum_{k=1}^m 2\xi_k\sin{k\theta}/\sqrt{k}$. To bound $H$ observe that 
\begin{equation}\label{compute}
\frac{e^{ikx}-e^{iky}}{2i\sin(\frac{x-y}{2})}= \sum_{l=1}^k e^{i(l-1/2)x}e^{i(k-l+1/2)y}    
\end{equation}
which gives
$$H(x,y)=\Big(\sum_{1\leq k\leq m} \frac{\xi_k}{\sqrt{k}}\sum_{l=1}^k e^{i(l-1/2)x}e^{i(k-l+1/2)y}+e^{-i(l-1/2)x}e^{-i(k-l+1/2)y} \Big)^2.$$
Thus,
$$ |H(x,y)|\leq \Big(\sum_{1 \leq k\leq m} 2\sqrt{k}|\xi_k| \Big)^2 \leq 2m(m+1)\|\xi\|^2 $$
by the Cauchy-Schwarz inequality. Next, by a Taylor expansion of $g$,
\begin{align*}
 \Big| g(\theta-i\frac{\nu}{n}h(\theta))-g(\theta)+\frac{i\nu}{n}h(\theta)g'(\theta) \Big| &\leq \frac{\nu^2h^2(\theta)}{2n^2} \sup_{t\in[0,1]} |g''(\theta-i\nu h(\theta)t/n)| \\
 &\leq \frac{\nu^2h^2(\theta)}{2n^2} \sup_{|t|\leq2\nu\sqrt{\log m +1}\|\xi\|/n} |g''(\theta+it)|
\end{align*}
since $\|h\|_\infty \leq 2\sqrt{\log m+1}\|\xi\|$ by the Cauchy-Schwarz inequality. Our hypothesis (\ref{nu1}) on the parameter $\nu$ gives
$$ \sup_{|t|\leq2\nu\sqrt{\log m +1}\|\xi\|/n} |g''(\theta+it)| < \sup_{|t|<c_0/m} \sum_{k=1}^m 2k^{3/2} |\xi_k|e^{k|t|} < e^{c_0}(m+1)^2\|\xi\|. $$
Therefore
\begin{align*}
|\exp (ig(\theta_j-i\frac{\nu}{n}h(\theta_j)))| &= |\exp( ig(\theta_j-i\frac{\nu}{n}h(\theta_j))-ig(\theta_j)-\frac{\nu}{n}h(\theta_j)g'(\theta_j)) \exp(\frac{\nu}{n}h(\theta_j)g'(\theta_j)| \\
&\leq \exp(\frac{e^{c_0}\nu^2}{2n^2}\|h\|_\infty^2(m+1)^2\|\xi\|)\exp(\frac{\nu}{n}h(\theta_j)g'(\theta_j)) \\
&\leq \exp(\frac{2e^{c_0}\nu^2}{n^2}(\log m +1)(m+1)^2\|\xi\|^3)\exp(\frac{\nu}{n}h(\theta_j)g'(\theta_j)).
\end{align*}
Finally, the Cauchy-Schwarz inequality tells us that $ \|h'\|_\infty \leq \sqrt{2m(m+1)}\|\xi\|$. All these estimates inserted in (\ref{bound}) yield the desired upper bound.

Next consider part (b), i.e. suppose that (\ref{nu2}) holds and set $h=g'$. To bound $H$ we can reuse (\ref{compute}): 
$$ H(x,y)=\Big(-\sum_{1\leq k\leq m} \sqrt{k}\xi_k \sum_{l=1}^k e^{i(l-1/2)x}e^{i(k-l+1/2)y}+e^{-i(l-1/2)x}e^{-i(k-l+1/2)y} \Big)^2$$
so by the Cauchy-Schwarz inequality,
$$ |H(x,y)|\leq \Big(\sum_{1 \leq k\leq m} 2k^{3/2} |\xi_k| \Big)^2 \leq (m+1)^4 \|\xi\|^2. $$
This time we Taylor expand $g$ up to second order, 
\begin{align*}
-\Im{\Big( g\Big(\theta+i\frac{\nu}{n}g'(\theta)\Big)-g(\theta)-\frac{i\nu}{n}g'(\theta)^2+\frac{\nu^2}{2n^2}g'(\theta)^2g^{''}(\theta) \Big)} &\leq \frac{\nu^3g'(\theta)^3}{6n^3} \sup_{t\in[0,1]} |g^{(3)}(\theta+i\nu g'(\theta)t/n)| \\
 &\leq \frac{\nu^3g'(\theta)^3}{6n^3} \sup_{|t|\leq \nu\sqrt{2m(m+1)}\|\xi\|/n } |g^{(3)}(\theta+it)|
\end{align*}
where we used that $\|g'\|_\infty \leq \sqrt{2m(m+1)}\|\xi\|$ from the Cauchy-Schwarz inequality. Inserting the definition of $\nu$ which was set in (\ref{nubis}) gives
$$  \sup_{|x|\leq \nu\sqrt{2m(m+1)}\|\xi\|/n } |g^{(3)}(\theta+ix)| < \sup_{|x|<\frac{\sqrt{2}\eta}{\sqrt{n}m}} \sum_{k=1}^m 2k^{5/2} |\xi_k|e^{k|x|} < \frac{\sqrt{2}e^{\sqrt{2/n}\eta}}{\sqrt{3}}(m+1)^3\|\xi\|. $$
Therefore,
\begin{align*}
|\exp(ig(\theta+i\frac{\nu}{n}g'(\theta)))| &= |\exp( ig(\theta+i\frac{\nu}{n}g'(\theta))-ig(\theta)+\frac{\nu}{n}g'(\theta)^2+i\frac{\nu^2}{2n^2}g'(\theta)^2g^{''}(\theta))\exp(-\frac{\nu}{n}g'(\theta)^2)| \\
&\leq \exp\Big(\frac{\nu^3}{n^3}\frac{\sqrt{2}e^{\sqrt{2/n}\eta}}{6\sqrt{3}}(m+1)^3\|\xi\|g'(\theta)^3-\frac{\nu}{n}g'(\theta)^2\Big) \\
&\leq \exp\Big(\frac{\nu^3}{n^3}\frac{e^{\sqrt{2/n}\eta}}{3\sqrt{3}}(m+1)^4\|\xi\|^2g'(\theta)^2-\frac{\nu}{n}g'(\theta)^2\Big) \\
&= \exp\Big(-\lambda g'(\theta)^2\Big).
\end{align*}
Finally, the Cauchy-Schwarz inequality gives $ \|g''\|_\infty \leq (m+1)^2 \|\xi\|$. If we combine all our estimates, then (\ref{bound}) becomes
\begin{equation*}
| F_{n,m}^{a,b}(\xi)| \leq e^{\nu^2(1-\frac{1}{2n})(m+1)^4\|\xi\|^2} \E_{O(2n)^+}[\prod_{1\leq j\leq n}e^{-\lambda g'(\theta_j)^2}].
\end{equation*}
Inserting the definition of $\nu$ finishes the proof.

Now consider the three other cases. The proposition follows by the same approach as for the first case except that now we also need to control the $n$ additional factors that appear in the probability densities (\ref{density2}), namely
\begin{align*}
\prod_{1\leq j \leq n}\sin^2(\theta_j),\quad
\prod_{1\leq j\leq n}\cos^2\Big(\frac{\theta_j}{2}\Big), \quad
\prod_{1\leq j\leq n}\sin^2\Big(\frac{\theta_j}{2}\Big).
\end{align*}
After the same change of variables as above, the first of these products can be bounded using Lemma \ref{ineq}:
\begin{align*}
\Big| \frac{\sin^2(\theta_j-i\nu h(\theta_j)/n)}{\sin^2(\theta_j)} \Big| = 1+\frac{\sinh^2(\nu h(\theta_j)/n)}{\sin^2(\theta_j)} \leq \exp{\Big(\frac{\nu h(\theta_j)}{n\sin(\theta_j)}\Big)^2}    
\end{align*}
for part (a) and
\begin{align*}
\Big| \frac{\sin^2(\theta_j+i\nu g'(\theta_j)/n)}{\sin^2(\theta_j)} \Big| = 1+\frac{\sinh^2(\nu g'(\theta_j)/n)}{\sin^2(\theta_j)} \leq \exp{\Big(\frac{\nu g'(\theta_j)}{n\sin(\theta_j)}\Big)^2}    
\end{align*}
for part (b).
Now, using that $|\sin(k\theta)/\sin(\theta)| = \lvert \sum_{j=0}^{k-1} e^{i(k-1-2j)\theta} \rvert \leq k$, we see that
\begin{align*}
\Big( \frac{h(\theta_j)}{\sin(\theta_j)} \Big)^2 \leq \Big( 2\sum_{k=1}^m \sqrt{k} |\xi_k| \Big)^2 \leq 2m(m+1)\|\xi\|^2    
\end{align*}
for part (a) and
\begin{align*}
\Big( \frac{g'(\theta_j)}{\sin(\theta_j)} \Big)^2 \leq \Big( 2\sum_{k=1}^m k^{3/2} |\xi_k| \Big)^2 \leq (m+1)^4 \|\xi\|^2    
\end{align*}
for part (b). These last four lines of inequalities explain the additional terms that appears in the bounds of $F_{n,m}^{++}$. A similar argument gives the same bounds for the last two cases.
\end{proof}

We will now use the Basor-Ehrhardt identities a second time to bound the expectation values appearing in part (a) of the previous proposition.
\begin{prop}\label{expectation}
Let $h$ be the Hilbert transform of $g$. Then, for $\xi\in\R^m$,
\begin{align*}
\E_{G(n)}\Big[\exp\Big(\frac{\nu}{n} \sum_{j=1}^n h(\theta_j)g'(\theta_j)\Big)\Big]\leq \exp\Big(-2\nu\|\xi\|^2 + 4\frac{\nu^2}{n^2}m^2(1+\log m)\|\xi\|^4\Big) \det(1+Q_nK_{G(n)}Q_n)
\end{align*}
where $K_{O(2n+1)^-}=H(a_+^{-1}\tilde{a_+})$, $K_{O(2n+1)^+}=-H(a_+^{-1}\tilde{a_+})$, $K_{Sp(2n)}=-H(e^{-i\theta}a_+^{-1}\tilde{a_+})$, $K_{O(2n)^+}=H(e^{i\theta}a_+^{-1}\tilde{a_+})$, and 
$a_+^{-1}\tilde{a_+}(e^{i\theta}) = \exp \Big(-2i\frac{\nu}{n}\sum_{k,j=1}^m \sqrt{\frac{k}{j}}\xi_k\xi_j(\sin (k+j)\theta -\sin \lvert k-j\rvert\theta)\Big)$.
\end{prop}
\begin{proof}
Take $G=O(2n+1)^-$ for simplicity; the proof of the result for the other groups uses the same arguments. According to Lemma \ref{Kurt} and Proposition \ref{BE}, 
\begin{align*}
&\E_{O(2n+1)^-}\Big[\exp\Big({\frac{\nu}{n} \sum_{j=1}^n h(\theta_j)g'(\theta_j)}\Big)\Big]= \\
&\qquad \exp\Big(n[\log a]_0 +\sum_{n=0}^\infty [\log a]_{2n+1}
+\frac{1}{2}\sum_{n=1}^\infty n[\log a]_n^2\Big) \det(1+Q_nH(a_+^{-1}\tilde{a_+})Q_n)
\end{align*}
where 
\begin{align}\label{a}
a(e^{i\theta})=\mathrm{exp}\Big(\frac{\nu}{n}h(\theta)g'(\theta)\Big)= \mathrm{exp}\Big(-4\frac{\nu}{n}\sum_{k,j=1}^m\sqrt{\frac{k}{j}}\xi_k\xi_j\sin k\theta \sin j\theta\Big)
\end{align}
and therefore
\begin{align*}
a_+^{-1}\tilde{a_+}(e^{i\theta}) = \exp \Big(-2i\frac{\nu}{n}\sum_{k,j=1}^m \sqrt{\frac{k}{j}}\xi_k\xi_j(\sin (k+j)\theta -\sin \lvert k-j\rvert\theta)\Big).    
\end{align*}
First observe that by the orthogonality of the sine function,
$$[\log a]_0 = \frac{\nu}{n}\int_0^{2\pi} g'(\theta)h(\theta) \frac{d\theta}{2\pi} = -2\frac{\nu}{n}\|\xi\|^2. $$
Next we see that the second term in the exponential, $\sum_{n=0}^\infty [\log a]_{2n+1}$, is equal to zero. Indeed, we have
$$ \log a(e^{i\theta}) = \frac{\nu}{n} \sum_{k,j=1}^m\sqrt{\frac{k}{j}}\xi_k\xi_j (e^{i(k+j)\theta}-e^{i(k-j)\theta}-e^{-i(k-j)\theta}+e^{-i(k+j)\theta})$$
so the sum of all Fourier coefficients of positive odd order will cancel out. Finally, to bound the last term in the exponential, we will use that for any real function $u$ on the unit circle that satisfies $\sum_{k=1}^\infty k\lvert\hat{u}_k\rvert^2 < \infty$, and with Hilbert transform $\mathcal{H}(u):= -i \sum_{j\in\Z} \mathrm{sgn}(j)\hat{u}_je^{ij\theta} $, the following holds
$$ \sum_{k=1}^\infty k\lvert\hat{u}_k\rvert^2= -\frac{1}{2}\int_0^{2\pi} u'(\theta)\mathcal{H}(u(\theta)) \frac{d\theta}{2\pi} \leq \frac{1}{2}\|u\|_{L_2}\|u'\|_{L_2}. $$
Thus, in our case, 
\begin{align*}
&\frac{1}{2}\sum_{k=1}^\infty k[\log a]_k^2 = \frac{\nu^2}{2n^2}\sum_{k=1}^\infty k[g'h]_k^2 \leq  \frac{\nu^2}{4n^2}\|g'h\|_{L^2}(\|g''h\|_{L^2}+\|g'h'\|_{L^2}) \\ &\leq \frac{\nu^2}{4n^2}\|h\|_\infty\|g'\|_{L^2}(\|h\|_\infty\|g''\|_{L^2}+\|h'\|_\infty\|g'\|_{L^2})\\
&< \sqrt{2}\frac{\nu^2}{n^2}\sqrt{m}\sqrt{1+\log m}(\sqrt{2}m^{3/2}\sqrt{1+\log{m}}+m\sqrt{m+1})\|\xi\|^4 \\
&< 4\frac{\nu^2}{n^2}m^2(1+\log m)\|\xi\|^4
\end{align*}
where the third inequality follows from the Cauchy-Schwarz inequality to bound the $L^\infty$ norms. This finishes the proof in the case of $O(2n+1)^-$.
\end{proof}

It remains to estimate the Fredholm determinants of Proposition \ref{expectation}. 
\begin{prop}\label{fred}
Let $m\geq2$, $m\leq n^{1/3}$, $\xi\in\R^m$. With $K_{G(n)}$ as in the previous proposition we have that
\begin{multline*}
\det(1+Q_nK_{G(n)}Q_n) \leq \exp\Big( \nu\frac{\sqrt{\log m+1}+1}{m^{5/3}}\frac{(1+m^{-1})^{4/3}}{1-m^{-3}}\Big(\frac{m^{-4/3}}{\sqrt{6}}(1+m^{-1})^{5/3}+1\Big)\|\xi\|^2 \\+ \nu^2\frac{(m+1)^{8/3}}{n^2}(\sqrt{\log m+1}+1)^2\|\xi\|^4 \Big)
\end{multline*} 
provided $\nu$ satisfies 
\begin{align}\label{nu2}
\frac{\nu}{n}(m+1)(1+\sqrt{\log m+1})\|\xi\|^2 <\frac{2n-1}{2(m+1)^{5/3}}.
\end{align}
\end{prop}

\begin{proof}
Recall that for any trace-class operator $K$, one can define the regularized determinant $\det_2$ by
$$ \det_2(I+K) = e^{-\tr  K}\det(I+K), $$
and prove that it satisfies the following inequality
$$ |\det_2(I+K)|\leq e^{\frac{1}{2}\|K\|^2_{\mathcal{J}_2}}, $$
which is actually valid for the larger class of Hilbert-Schmidt operators (Theorem 9.2 in \cite{Simon}). By Theorem \ref{BE} all our $K_{G(n)}$ are trace-class, thus
\begin{equation}\label{fourier0}
|\det(1+Q_nK_{G(n)}Q_n)|\leq e^{ |\tr  Q_nK_{G(n)}Q_n|+\frac{1}{2}\|Q_nK_{G(n)}Q_n\|^2_{\mathcal{J}_2}}.
\end{equation}
Since $K_{G(n)}$ is a certain type of Hankel operator with symbol $a_+^{-1}\tilde{a_+}$, we need bounds on the Fourier coefficients of $a_+^{-1}\tilde{a_+}$. We have 
$$ a_+^{-1}\tilde{a_+}(e^{i\theta}) = \exp \Big(-2i\frac{\nu}{n}\sum_{k,j=1}^m \sqrt{\frac{k}{j}}\xi_k\xi_j(\sin (k+j)\theta -\sin \lvert k-j\rvert\theta)\Big). $$
Integrating by parts its $l$th Fourier coefficient twice and taking the absolute value gives
\begin{align}
|[a_+^{-1}\tilde{a_+}]_l| &\leq \frac{4\nu}{nl^2} \sum_{k,j=1}^m \sqrt{\frac{k}{j}}\lvert\xi_k\xi_j\rvert(k^2+j^2) +\frac{4\nu^2}{n^2l^2} \Big(\sum_{k,j=1}^m \sqrt{\frac{k}{j}}\lvert\xi_k\xi_j\rvert((k+j)+\lvert k-j\rvert)\Big)^2
\end{align}
which we can bound using the Cauchy-Schwarz inequality:
\begin{align}
|[a_+^{-1}\tilde{a_+}]_l| &\leq \frac{4\nu}{\sqrt{6}nl^2}(m+1)^3(\sqrt{\log m+1}+1)\|\xi\|^2+\frac{4\nu^2}{n^2l^2}(m+1)^4(\sqrt{\log m+1}+1)^2\|\xi\|^4.
\end{align}
Thus, by (\ref{nu2}),
$$
|[a_+^{-1}\tilde{a_+}]_l| \leq \frac{4\nu}{nl^2}(m+1)^3(\sqrt{\log m+1}+1)\Big(\frac{1}{\sqrt{6}}+\frac{(2n-1)}{2(m+1)^{5/3}}\Big)\|\xi\|^2 := B_l.
$$
Now, by definition of the kernels $K_{G(n)}$,
\begin{align*}
&\tr  Q_nK_{O(2n+1)^-}Q_n = \sum_{j\geq n} [a_+^{-1}\tilde{a_+}]_{2j+1}, &&
\tr  Q_nK_{O(2n+1)^+}Q_n = \sum_{j\geq n} [a_+^{-1}\tilde{a_+}]_{2j+1}, \\
&\tr  Q_nK_{O(2n)^-}Q_n = \sum_{j\geq n} [a_+^{-1}\tilde{a_+}]_{2j+2}, &&
\tr  Q_nK_{O(2n)^+}Q_n = \sum_{j\geq n} [a_+^{-1}\tilde{a_+}]_{2j} 
\end{align*}
so because $B_j$ is decreasing it suffices to estimate
\begin{align}
 \sum_{j\geq n} |B_{2j}|  
&\leq \frac{\nu}{n}\frac{(m+1)^3}{(n-1)}(\sqrt{\log m+1}+1)\Big(\frac{1}{\sqrt{6}}+\frac{(2n-1)}{2(m+1)^{5/3}}\Big)\|\xi\|^2\\
&= \nu(\sqrt{\log m+1}+1)\Big(\frac{(m+1)^3}{\sqrt{6}n(n-1)}+\frac{(2n-1)(m+1)^{4/3}}{2n(n-1)}\Big)\|\xi\|^2  \\
&\leq \nu\frac{\sqrt{\log m+1}+1}{m^{5/3}}\frac{(1+m^{-1})^{4/3}}{1-m^{-3}}\Big(\frac{m^{-4/3}}{\sqrt{6}}(1+m^{-1})^{5/3}+1\Big)\|\xi\|^2
\end{align}
where we used that $m\leq n^{1/3}$. 

Similarly, 
\begin{align*}
& \|Q_nK_{O(2n+1)^-}Q_n\|^2_{\mathcal{J}_2} = \sum_{j\geq 2n} (j-2n+1) [a_+^{-1}\tilde{a_+}]_{j+1}^2, &&
\|Q_nK_{O(2n+1)^+}Q_n\|^2_{\mathcal{J}_2} = \sum_{j\geq 2n} (j-2n+1) [a_+^{-1}\tilde{a_+}]_{j+1}^2, \\
&\|Q_nK_{O(2n)^-}Q_n\|^2_{\mathcal{J}_2} = \sum_{j\geq 2n} (j-2n+1) [a_+^{-1}\tilde{a_+}]_{j+2}^2, &&
\|Q_nK_{O(2n)^+}Q_n\|^2_{\mathcal{J}_2} = \sum_{j\geq 2n} (j-2n+1) [a_+^{-1}\tilde{a_+}]_{j}^2 
\end{align*}
so we can restrict our attention to
\begin{align*}
\sum_{j\geq 2n} (j-2n+1) B_j^2 &\leq \frac{4\nu^2}{n^2(2n-1)^2}(m+1)^6(\sqrt{\log m+1}+1)^2\Big(\frac{1}{\sqrt{6}}+\frac{(2n-1)}{2(m+1)^{5/3}}\Big)^2\|\xi\|^4
\end{align*}
where we used the bound
\[ \sum_{j\geq2n} \frac{j-2n+1}{j^4} \leq \int_{2n-1}^\infty \frac{\d x}{x^3} - (2n-1) \int_{2n}^\infty \frac{\d x}{x^4} = \frac{1}{2(2n-1)^2}-\frac{2n-1}{3(2n)^3} \leq \frac{\Big(\frac{1}{2}-\frac{15^3}{3\cdot16^3}\Big)}{(2n-1)^2} < \frac{1}{4(2n-1)^2} \]
since $m\leq n^{1/3}$ and $m\geq2$. These assumptions also give
\begin{align*}
\|Q_nK_{G(n)}Q_n\|^2_{\mathcal{J}_2}
&\leq 
\frac{\nu^2(\sqrt{\log m+1}+1)^2}{n^2}\Big(\frac{\sqrt{2}(m+1)^3}{\sqrt{3}(2m^3-1)}+(m+1)^{4/3}\Big)^2\|\xi\|^4 \\
&\leq \frac{\nu^2(m+1)^{8/3}}{n^2}(\sqrt{\log m+1}+1)^2\Big(\frac{2\cdot 3^{1/6}}{5}+1\Big)^2\|\xi\|^4 \\
&< 2\frac{\nu^2(m+1)^{8/3}}{n^2}(\sqrt{\log m+1}+1)^2\|\xi\|^4.
\end{align*}
These bounds inserted in (\ref{fourier0}) give the desired inequality.

\end{proof}

Combining Propositions \ref{start}, \ref{expectation} and \ref{fred} we see that if $\nu$ satisfies (\ref{nu1}) and (\ref{nu2}), and if $m\geq2$, $m\leq n^{1/3}$, then
\begin{align}\label{final-short}
\lvert F_{n,m}^{a,b}(\xi)\rvert &\leq \exp(-\alpha\nu+\delta\nu^2)
\end{align}
where 
$$\alpha= 2\|\xi\|^2 - \frac{\sqrt{\log m+1}+1}{m^{5/3}}\frac{(1+m^{-1})^{4/3}}{1-m^{-3}}\Big(\frac{m^{-4/3}}{\sqrt{6}}(1+m^{-1})^{5/3}+1\Big)\|\xi\|^2$$
and 
\begin{align*}
\delta&=4\frac{m^2}{n^2}(1+\log m)\|\xi\|^4 +\Big(2+\frac{1}{n}\Big)m(m+1)\|\xi\|^2+\frac{2e^{c_0}}{n}(m+1)^2(\log m+1)\|\xi\|^3 \\ &+ \frac{(m+1)^{8/3}}{n^2}(\sqrt{\log m+1}+1)^2\|\xi\|^4,
\end{align*}
for all pairs $(a,b)=(\pm1/2,\pm1/2)$ and any $\xi\in\R^m$. Optimizing over $\nu$ yields $\nu=\alpha/2\delta$ which is positive for $m\geq 2$. We obtain

\begin{prop}\label{regime2}
Let $m\leq n^3$, $m\geq 2$. Then, for any pair $(a,b)=(\pm1/2,\pm1/2)$ and any $\xi\in\R^m$,
\begin{equation}\label{bound2}
\lvert F_{n,m}^{ab}(\xi)\rvert \leq \exp\Big(-\frac{(1-c_1(m))^2}{c_2(m)}\frac{ n^2 \wedge \|\xi\|^2}{(m+1)^{8/3}(\log m+1)}\Big)
\end{equation}
where 
\begin{equation}\label{c1}
c_1(m)=\frac{(1+m^{-1})^{4/3}}{2(1-m^{-3})}\Big(1+\frac{m^{-4/3}}{\sqrt{6}}(1+m^{-1})^{5/3}\Big)\frac{\sqrt{\log m+1}+1}{m^{5/3}}    
\end{equation}
and 
\begin{multline}\label{c2}
    c_2(m) = \Big(4m^2(\log m+1)+2e^\frac{1}{3}(m+1)^2(\log m+1) + \Big(2+\frac{1}{m^3}\Big)m(m+1)+(m+1)^{8/3}(\sqrt{\log m+1}+1)^2 \Big) \\ (m+1)^{-8/3}(\log m+1)^{-1} 
\end{multline}      
Consequently, for any $\Lambda_2\geq\Lambda_1$ with $\Lambda_1$ given by (\ref{lambda1}),
\begin{multline}\label{regime2bis}
 \int_{\Lambda_1\leq\xi\leq\Lambda_2} \lvert F_{n,m}^{ab}(\xi)\rvert^2 d\xi  \\
\leq\Omega_m n^m \exp\Big(\frac{-(1-c_1(m))^2 n^2}{2c_2(m)m^2(m+1)^{8/3}(\log m+1)^2}\Big) + \Omega_m \Lambda_2^m  \exp\Big(\frac{-2(1-c_1(m))^2 n^2}{c_2(m)(m+1)^{8/3}(\log m+1)}\Big).
\end{multline}
\end{prop}

\begin{proof}
Before inserting $\nu=\alpha/2\delta$ in (\ref{final-short}) we need to check that it satisfies (\ref{nu1}) and (\ref{nu2}). For the first inequality, we have
\begin{align*}
\frac{\alpha}{2\delta}< \frac{\|\xi\|^2}{\delta} <\frac{n}{2e^{c_0}(m+1)^2(\log m+1)\|\xi\|} 
\end{align*}
which gives
\begin{align*}
2\frac{\nu}{n}m\sqrt{\log m +1}\|\xi\| < \frac{1}{e^{c_0}(m+1)\sqrt{\log m+1}} \leq \frac{1}{e^{c_0}3\sqrt{\log2+1}}   
\end{align*}
if $m\geq2$, and this will be less than $c_0$ if we simply choose $c_0=1/3$. For the second inequality we can use that
\begin{align*}
\frac{\alpha}{2\delta} < \frac{\|\xi\|^2}{\delta} <\frac{n^2}{(m+1)^{8/3}(\sqrt{\log m+1}+1)^2\|\xi\|^2}  
\end{align*}
and therefore
\begin{align*}
\frac{\nu}{n}(m+1)(1+\sqrt{\log m+1})\|\xi\|^2 &< \frac{n}{(m+1)^{5/3}(\sqrt{\log m+1}+1)}  \\
& < \frac{1}{2}\frac{2n-1}{(m+1)^{5/3}}.
\end{align*}
for all $m\geq2$. Hence (\ref{final-short}) becomes 
\begin{equation} \label{final-long}
\lvert F_{n,m}^{a,b}(\xi)\rvert \leq \exp{\Big(-\frac{\alpha^2}{4\delta}}\Big) = \exp\Big(\frac{-(1-c_1(m))^2\|\xi\|^2}{\delta/\|\xi\|^2}\Big).
\end{equation}
for all $m\geq 2$. Now replace either $\|\xi\|/n$ or $n/\|\xi\|$ by one depending on whether $\|\xi\|\leq n$ or $n\leq\|\xi\|$,
\begin{multline*}
\lvert F_{n,m}^{a,b}(\xi)\rvert \leq \exp\Big( -(1-c_1(m))^2 n^2 \wedge \|\xi\|^2 \\ 
\cdot \Big(4m^2(1+\log m)+2e^\frac{1}{3}(m+1)^2(\log m+1) + \Big(2+\frac{1}{n}\Big)m(m+1)+(m+1)^{8/3}(\sqrt{\log m+1}+1)^2\Big)^ {-1}\Big).    
\end{multline*}
The denominator in the exponential satisfies, for $n\geq m^{3}$,
\begin{multline*}
4m^2(1+\log m)+2e^\frac{1}{3}(m+1)^2(\log m+1) + \Big(2+\frac{1}{n}\Big)m(m+1)+(m+1)^{8/3}(\sqrt{\log m+1}+1)^2 \\ < c_2(m) (m+1)^{8/3}(\log m+1).    
\end{multline*}
This gives (\ref{bound2}). We can now bound the $L_2$-norm by writing
\begin{multline*}
\int_{\Lambda_1\leq \|\xi\| \leq\Lambda_2} \lvert F_{n,m}^{ab}(\xi)\rvert^2 d\xi  =  \int_{\Lambda_1\leq \|\xi\| \leq n} \lvert F_{n,m}^{ab}(\xi)\rvert^2 d\xi  + \int_{n\leq \|\xi\| \leq\Lambda_2} \lvert F_{n,m}^{ab}(\xi)\rvert^2 d\xi \\
< \Omega_m n^m \exp\Big(\frac{-2(1-c_1(m))^2 \Lambda_1^2}{c_2(m)(m+1)^{8/3}(\log m+1)}\Big) + \Omega_m \Lambda_2^m  \exp\Big(\frac{-2(1-c_1(m))^2 n^2}{c_2(m)(m+1)^{8/3}(\log m+1)}\Big).
\end{multline*}
which is (\ref{regime2bis}) if we replace $\Lambda_1$ by its definition, $\Lambda_1=n/(2m\sqrt{\log m +1})$.
\end{proof}

\section{Large regime}
For the last regime we need a bound on the characteristic function $F_{n,m}^{a,b}$ that decays with $\xi$ since we eventually integrate it over all $\xi\in\R^m$. Our method relies on the change of variables of the previous section, i.e. part (b) of Proposition \ref{start} is our starting point, but instead of using the Basor-Ehrhardt formulas we apply the following lemma.

\begin{lemma}\label{hadamard}
For any pair $(a,b)=(\pm\frac{1}{2},\pm\frac{1}{2})$, the joint eigenvalue probability density satisfies
$$\sup_{\theta\in[0,\pi]^n} |\rho_n^{a,b}(\theta)| \leq \frac{(2e/\pi)^n}{\sqrt{2\pi n}}. $$
\end{lemma}

\begin{proof}
One can show (proof of Proposition 3.7 in \cite{Meckes} or Exercise 5.5.4 in \cite{Forrester})
\begin{align*}
\rho_n^{--}(\theta)&  = \frac{2^n}{2n!\pi^n}(\det[\cos(k-1)\theta_j]_{1\leq j,k\leq n})^2 \\ 
\rho_n^{++}(\theta)&= \frac{2^n}{n!\pi^n}(\det[\sin k\theta_j]_{1\leq j,k\leq n})^2 \\ 
\rho_n^{-+}(\theta)&=\frac{2^n}{n!\pi^n}(\det[\cos (k-1/2)\theta_j]_{1\leq j,k\leq n})^2 \\ 
\rho_n^{+-}(\theta)&= \frac{2^n}{n!\pi^n}(\det[\sin(k-1/2)\theta_j]_{1\leq j,k\leq n})^2 
\end{align*}
The result follows by applying Hadamard's formula to each determinant and Stirling's approximation (Inequality (\ref{Stirling})) to $n!$.
\end{proof}

This shows that the expected values appearing in part (b) of Proposition \ref{start} satisfy
\begin{align}\label{hadamard2}
\E_{G(n)}[\prod_{1\leq j\leq n} e^{-\lambda g'(\theta_j)^2} ] \leq \frac{(2e)^n}{\sqrt{2\pi n}}\Big(\frac{1}{\pi} \int_0^{\pi} e^{-\lambda g'(x)^2} dx \Big)^n.
\end{align}
To evaluate the integral on the right-hand side we will need the following result, obtained in \cite{Chahkiev}.

\begin{lemma}\label{chahkiev}
Let $p_m$ be a trigonometric polynomial given by 
$$p_m(\theta) = \frac{a_0}{2}+\sum_{k=1}^m a_k\cos{k\theta}+b_k\sin{k\theta},$$ 
where $a_k$, $b_k$ are real. Define $G(t)= \frac{1}{2\pi}\mu\{e^{i\theta}\in\mathbf{T}, |p_m(\theta)|\leq t \}$, where $\mu$ denotes Lebesgue measure on the unit circle $\mathbf{T}$. Then,
\begin{equation}\label{G_bound}
G(t) \leq 2e\Big(\frac{t}{\sqrt{2}\|p_m\|_2}\Big)^\frac{1}{2m}.
\end{equation}
\end{lemma}

\begin{prop}\label{regime3}
For any pair $(a,b)=(\pm\frac{1}{2},\pm\frac{1}{2})$ and any $\Lambda_2>0$ we have that
\begin{equation}
\int_{\Lambda_2 \leq \|\xi\|} |F_{n,m}^{a,b}(\xi)|^2 d\xi \leq  \frac{(2e)^{4n}}{2\pi n} (c_3(m)\sqrt{nm}m^2)^{\frac{N}{2}}  m\Omega_{m} \frac{\Lambda_2^{m-N/2}}{N/2-m}
\end{equation}
provided $m\geq3$, $n\geq m^3$, and where
\begin{equation}\label{c3}
    c_3(m) = \frac{e^{\frac{1}{2}(1+\frac{1}{2m^3})}(1+m^{-1})^2}{\sqrt{2}(1-\frac{e^{1/2m^2}}{24\sqrt{3}m^4}) }.
\end{equation}
\end{prop}
\begin{proof}
Define $G(t)$ as in Lemma \ref{chahkiev}, with $p_{m}(\theta)=g'(\theta)$. Then
\begin{equation}\label{int1}
\frac{1}{\pi} \int_0^{\pi} e^{-\lambda g'(\theta)^2} d\theta = \frac{1}{2\pi} \int_0^{2\pi} e^{-\lambda g'(\theta)^2} d\theta =  \int_0^\infty e^{- t}G\Big(\sqrt{\frac{t}{\lambda}}\Big) dt.
\end{equation}
Estimate (\ref{G_bound}) becomes
$$G\Big(\sqrt{\frac{t}{\lambda}}\Big) \leq 2e\Big(\frac{t}{2\lambda \|g'\|_2^2}\Big)^{1/4m} $$
which inserted in (\ref{int1}) gives
\begin{multline*}
\frac{1}{\pi} \int_0^{\pi}  e^{-\lambda g'(\theta)^2} d\theta \leq 2e\Big(\frac{1}{2\lambda \|g'\|_2^2}\Big)^{1/4m} \int_0^\infty e^{-t}t^{1/4m} dt \\
= 2e\Big(\frac{1}{2\lambda \|g'\|_2^2}\Big)^{1/4m}\Gamma\Big(1+\frac{1}{4m}\Big) < 2e\Big(\frac{1}{2\lambda \|g'\|_2^2}\Big)^{1/4m}.    
\end{multline*}
Recall that
\begin{align*}
g'(\theta) = -2\sum_{k=1}^m \sqrt{k} \xi_k \sin k\theta.
\end{align*}
Hence
\begin{align*}
\|g'\|_2^2 = 2 \sum_{k=1}^m k\xi_{k}^2 \geq 2 \sum_{k=1}^m \xi_{k}^2  = 2\|\xi\|^2,
\end{align*}
and by (\ref{hadamard2}),
$$
\E_{G(n)}[\prod_{1\leq j\leq n} e^{-\lambda g'(\theta_j)^2} ] \leq \frac{(2e)^n}{\sqrt{2\pi n}}\Big(\frac{1}{\pi} \int_0^{\pi} e^{-\lambda g'(x)^2} dx \Big)^{n} \leq   \frac{(2e)^{2n}}{\sqrt{2\pi n}} \Big(\frac{1}{4\lambda \|\xi\|^2}\Big)^{n/4m}.   
$$
It now follows from Proposition \ref{start}, part (b), that
\begin{equation*}
| F_{n,m}^{a,b}(\xi)| \leq e^{\eta^2(n+\frac{1}{2})} \frac{(2e)^{2n}}{\sqrt{2\pi n}} \Big(\frac{1}{4\lambda \|\xi\|^2}\Big)^{n/4m}   
\end{equation*}
Inserting the definition of $\lambda$ gives
\begin{equation*}
| F_{n,m}^{a,b}(\xi)| \leq \frac{(2e)^{2n}}{\sqrt{2\pi n}} \Big(\frac{e^{4m\eta^2(1+\frac{1}{2n})}\sqrt{n}(m+1)^2}{4\eta(1-\frac{\eta^2e^{\sqrt{2/n}\eta}}{3\sqrt{3}n}) \|\xi\|}\Big)^{n/4m}.   
\end{equation*}
We now choose $\eta\in[0,1]$ to be the minimizer of $e^{4m\eta^2}/\eta$, i.e. we set $\eta=(2\sqrt{2m})^{-1}$ and obtain
\begin{equation*}
| F_{n,m}^{a,b}(\xi)| \leq \frac{(2e)^{2n}}{\sqrt{2\pi n}} \Big(\frac{e^{\frac{1}{2}(1+\frac{1}{2n})}\sqrt{nm}(m+1)^2}{\sqrt{2}(1-\frac{e^{1/2\sqrt{nm}}}{24\sqrt{3}nm}) \|\xi\|}\Big)^{n/4m}.   
\end{equation*}
Thus
\begin{align}\label{lambda20}
 \int_{\Lambda_2 \leq \|\xi\|} |F_{n,m}^{a,b}(\xi)|^2 d\xi  \leq \frac{(2e)^{4n}}{2\pi n} \Big(\frac{e^{\frac{1}{2}(1+\frac{1}{2n})}\sqrt{nm}(m+1)^2}{\sqrt{2}(1-\frac{e^{1/2\sqrt{nm}}}{24\sqrt{3}nm}) }\Big)^{\frac{N}{2}} \int_{\Lambda_2 \leq \|\xi\|} \|\xi\|^{-\frac{N}{2}} d\xi. 
\end{align}
A change of variables to spherical coordinates gives
\begin{equation*}
\int_{\Lambda_2 \leq \|\xi\|} \|\xi\|^{-\frac{N}{2}} d\xi = S_{m-1} \int_{\Lambda_2}^{\infty} r^{m-1-\frac{N}{2}} dr = S_{m-1} \frac{\Lambda_2^{m-\frac{N}{2}}}{\frac{N}{2}-m}    
\end{equation*}
where $S_{m-1}$ is the surface area of the $m-1$-dimensional unit sphere, and where we used that $m<\frac{N}{2}$ which follows from $m\geq 3$ and $n\geq m^3$. Observing that $S_{m-1}=m\Omega_m$, using that $n\geq m^3$ and inserting the definition of $c_3(m)$ in (\ref{lambda20}) prove the statement.
\end{proof}

\section{Proof of the main Theorem}

In this section we first combine all our estimates to prove the bound on the $L_2$ distance between our random vector and a standard normal one, given in Theorem \ref{L2}. The total variation bound from Theorem \ref{L1} then follows from a result on tail probabilities. We conclude with three corollaries which give simple bounds for some special cases of $m$ and $n$.

\begin{proof}[Proof of Theorem \ref{L2}]
By Plancherel's theorem, since $p_{n,m}^{a,b}-\Psi_{n,m}\in L^2(\R^m)$,
\begin{equation*}
\|p_{n,m}^{a,b}-\Psi_{n,m} \|_2 = \|F_{n,m}^{a,b}-e^{-\|\cdot\|^2/2}\|_2.    
\end{equation*}
The triangle inequality gives
\begin{multline*}
\|F_{n,m}^{a,b}-e^{-\|\cdot\|^2/2}\|_2 \leq \|(F_{n,m}^{a,b}-e^{-\|\cdot\|^2/2})\mathbb{1}\{\|\xi\|\leq\Lambda_1\}\|_2 + \|F_{n,m}^{a,b}\mathbb{1}\{\Lambda_1 \leq \|\xi\|\}\|_2 +  \|e^{-\|\cdot\|^2/2}\mathbb{1}\{\Lambda_1 \leq \|\xi\| \} \|_2
\end{multline*}
where $\Lambda_1$ is given by (\ref{lambda1}). By Proposition \ref{regime1}, \ref{regime2} and \ref{regime3},
\begin{gather*}
\|(F_{n,m}^{a,b}-e^{-\|\cdot\|^2/2})\mathbb{1}\{\|\xi\|\leq\Lambda_1\}\|_2 \leq \frac{16}{15}e^{13/24}(e^{9/8}+1)  \frac{m^{3/2}\sqrt{\Omega_m}}{\sqrt{N}} \Big(\frac{m}{2}\Big)^{\frac{m}{4}}\frac{(e^{3/2}(\log m +1))^{N}}{\sqrt{\Gamma(2N+1)}}, \\
\|F_{n,m}^{a,b}\mathbb{1}\{\Lambda_1 \leq \|\xi\|\}\|_2^2 \leq \Omega_m n^m \exp\Big(\frac{-(1-c_1(m))^2 n^2}{2c_2(m)m^2(m+1)^{8/3}(\log m+1)^2}\Big) +\\ \Omega_m \Lambda_2^m  \exp\Big(\frac{-2(1-c_1(m))^2 n^2}{c_2(m)(m+1)^{8/3}(\log m+1)}\Big) +
 \frac{(2e)^{4n}}{2\pi n} (c_3(m)\sqrt{nm}m^2)^{\frac{N}{2}} m\Omega_m \frac{\Lambda_2^{m-N/2}}{N/2-m}
\end{gather*}
so it remains to choose $\Lambda_2$ and to estimate $\|e^{-\|\cdot\|^2/2}\mathbb{1}\{\Lambda_1 \leq \|\xi\| \} \|_2$. For the first task, set
\[ \epsilon =  \exp\Big(\frac{-2(1-c_1(m))^2 n^2}{c_2(m)(m+1)^{8/3}(\log m+1)}\Big), \quad C= \frac{(2e)^{4n}}{2\pi n} (c_3(m)\sqrt{nm}m^2)^{\frac{N}{2}}   \frac{2m}{N-2m}.\]
The minimum of $\Omega_m\Lambda_2^m(\epsilon+C\Lambda_2^{-N/2})$ is attained when $\Lambda_2^{N/2} = \frac{C}{\epsilon} (\frac{N}{2m}-1)$ and equals
\[\Omega_mC^{\frac{2m}{N}}\frac{N}{N-2m}(\frac{N}{2m}-1)^{\frac{2m}{N}}\epsilon^{1-\frac{2m}{N}}.\]
Inserting the values of $\epsilon$ and $C$, taking the square root and recalling that $m\geq3$, $n\geq m^3$ yields
\begin{multline*}
\|F_{n,m}^{a,b}\mathbb{1}\{\Lambda_1 \leq \|\xi\|\}\|_2 \leq  \sqrt{3\Omega_m}N^\frac{m}{4} (2e)^{4m^2} \Big(\frac{\sqrt{c_3(m)}m^\frac{3}{2}}{(2\pi n)^{1/N}}\Big)^m \exp\Big(-\frac{(1-c_1(m))^2 n^2}{3c_2(m)(m+1)^{8/3}(\log m+1)}\Big) \\
+ \sqrt{\Omega_m} n^\frac{m}{2} \exp\Big(-\frac{(1-c_1(m))^2 n^2}{4c_2(m)m^2(m+1)^{8/3}(\log m+1)^2}\Big).
\end{multline*}

For the second task we make a change of variables to spherical coordinates
\begin{equation*}
\|e^{-\|\cdot\|^2/2}\mathbb{1}\{\Lambda_1 \leq \|\xi\| \} \|_2^2 =  \int_{\|\xi\|>\Lambda_1} e^{-\|\xi\|^2} d\xi 
= S_{m-1} \int_{\Lambda_1}^\infty r^{m-1}e^{-r^2} dr 
= \frac{S_{m-1}}{2} \int_{\Lambda_1^2}^\infty r^{\frac{m}{2}-1}e^{-r} dr.
\end{equation*}
Repeated integration by parts (or 8.8.10 in \cite{DLMF}) gives
$$ \int_{\Lambda_1^2}^\infty r^{\frac{m}{2}-1}e^{-r} dr = e^{-\Lambda_1^2} \Gamma\Big(\frac{m}{2}\Big) \sum_{j=0}^{\frac{m}{2}-1} \frac{\Lambda_1^{2(\frac{m}{2}-1-j)}}{\Gamma(\frac{m}{2}-j)} $$
if $m$ is even, and
$$ \int_{\Lambda_1^2}^\infty r^{\frac{m}{2}-1}e^{-r} dr = \frac{\Gamma(\frac{m}{2})}{\Gamma(-\frac{1}{2})}\Gamma(-1/2,\Lambda_1^2)+ e^{-\Lambda_1^2} \Gamma\Big(\frac{m}{2}\Big) \sum_{j=0}^{\frac{m-1}{2}} \frac{\Lambda_1^{2(\frac{m}{2}-1-j)}}{\Gamma(\frac{m}{2}-j)}. $$
if $m$ is odd. Here $\Gamma(a,z)$ is the incomplete Gamma function. Observe that $\Gamma(-1/2)<0$ so for any $m\in\N$,
\begin{align*}
 \int_{\Lambda_1^2}^\infty r^{\frac{m}{2}-1}e^{-r} dr &\leq e^{-\Lambda_1^2} \Gamma\Big(\frac{m}{2}\Big) \sum_{j=0}^{\ceil{\frac{m}{2}}-1} \frac{\Lambda_1^{2(\frac{m}{2}-1-j)}}{\Gamma(\frac{m}{2}-j)}   \\
 &\leq e^{-\Lambda_1^2} \sum_{j=0}^{\ceil{\frac{m}{2}}-1} \Lambda_1^{2(\frac{m}{2}-1-j)}\Big(\frac{m}{2}-1\Big)^j \\
 &< e^{-\Lambda_1^2} \frac{\Lambda_1^m}{\Lambda_1^2-m/2+1}
\end{align*}
where we used that for any $j, m\in\N$, $\Gamma\Big(\frac{m}{2}\Big) \leq \Gamma\Big(\frac{m}{2}-j\Big)\Big(\frac{m}{2}-1\Big)^j$ (which follows from the recurrence relation $\Gamma(z)=z\Gamma(z-1)$) and that $\Lambda_1^2>m/2-1$ (by definition (\ref{lambda1}) of $\Lambda_1$ and because $n\geq m^3$). We obtain, again by the definition of $\Lambda_1$ and because $n\geq m^3$, $m\geq 3$, 
\begin{align}
\|e^{-\|\cdot\|^2/2}\mathbb{1}\{\Lambda_1 \leq \|\xi\| \} \|_2^2 &<  \frac{m\Omega_m\Lambda_1^me^{-\Lambda_1^2}}{2(\Lambda_1^2-m/2+1)}  \\
&\leq \frac{2(\log m +1)m\Omega_m}{N^2(1-2(\log m+1)(m-2)/m^4)} \Big(\frac{N}{2\sqrt{\log m +1}}\Big)^m e^{-\frac{N^2}{4(\log m+1)}} \\
&\leq  \frac{\Omega_m N^m}{(2\sqrt{\log m +1})^{m-2}}\frac{m}{N^2} e^{ -\frac{N^2}{4(\log m+1)}}.
\end{align}
\end{proof}

Theorem \ref{L2} will give us the bound on the total variation when combined with the following result on tail probabilities.

\begin{lemma}\label{largedev}
Assume $L>\frac{2\sqrt{6}m^2}{\sqrt{n-1}}$ and $m\geq 4$. Let $\square_L = [-\frac{L}{2}, \frac{L}{2}]^{m}$. Then, if $n\geq m^4$,
\begin{align*}
\mathrm{P}_{n}^{a,b} [\mathbf{X} \notin \square_L] \leq 2m e^{-\frac{L^2}{48m}}
\end{align*}
and if $n\geq m^3$,
\begin{align*}
\mathrm{P}_{n}^{a,b} [\mathbf{X} \notin \square_L] \leq 2m e^{-\frac{L^2}{80m}}.
\end{align*}
\end{lemma}
\begin{proof}
We prove the case $(a,b)=(-1/2,-1/2)$. By Lemma \ref{Kurt}, for any $\lambda>0$,
\[ \E_{O(2n)^+} [e^{\lambda(\tr U^k-\E\tr U^k)}] = e^{-\lambda\E\tr U^k}\det (\hat{g}_{j-k}+\hat{g}_{j+k})_{0\leq i,j \leq n-1} \]
where $g(\theta) = e^{2\lambda\cos(k\theta)}$.
Therefore, the assumptions of Proposition \ref{BE} are met and we obtain (recall that $\E_{O(2n)^+} \tr U^k = \eta_k$)
\[ \E_{O(2n)^+} [e^{\lambda(\tr U^k-\E\tr U^k)}]= e^{\frac{k}{2}\lambda^2}\det(1+Q_nH(e^{i\theta}e^{-2i\lambda\sin k\theta})Q_n). \]
To bound the Fredholm determinant we use that for any trace class operator K given by the infinite matrix $(K_{ij})_{i,j=1}^\infty$,
\[ \det(1+K)\leq e^{\|K\|_{J_1}} \leq e^{\sum_{i\geq1} (\sum_{j\geq 1} |K_{ij}|^2)^{1/2}} \]
see Theorem II.3.3 and exercise II.21 in \cite{GGK}. Two partial integrations give
\[ |\widehat{(e^{-2i\lambda\sin(k\cdot)}})_l| \leq \frac{1}{l^2} (4\lambda^2+2\lambda)k^2. \]
Moreover,
\[ \sum_{i\geq n} \Big( \sum_{j\geq n} \frac{1}{(i+j)^4} \Big)^{1/2} \leq \int_{n-1}^\infty \Big( \int_{n-1}^\infty \frac{\d x}{(x+y)^4} \Big)^{1/2} \d y = \frac{\sqrt{2}}{\sqrt{3(n-1)}}. \]
Thus,
\begin{align*}
\det(1+Q_nH(e^{i\theta}e^{-2i\lambda\sin k\theta})Q_n) &\leq \exp\Big(\sum_{i\geq n} \Big(\sum_{j\geq n} |\widehat{(e^{-2i\lambda\sin(k\cdot)}})_{i+j}|^2\Big)^{1/2}\Big)  \\
&\leq \exp(\frac{2\sqrt{2}}{\sqrt{3(n-1)}}(2\lambda^2+\lambda)k^2).
\end{align*}
Now, by Markov's inequality,
\[
\mathrm{P}_n^{--} [|\tr U^k-\E\tr U^k |\geq L] \leq e^{-\lambda L}(\E_{O(2n)^+} [e^{\lambda(\tr U^k-\E\tr U^k)}]+\E_{O(2n)^+} [e^{-\lambda(\tr U^k-\E\tr U^k)}]). \]
Inserting the above estimates gives
\[\mathrm{P}_n^{--} [|\tr U^k-\E\tr U^k |\geq L] \leq 2 \exp \Big(-\lambda \Big(L - \frac{2\sqrt{2}k^2}{\sqrt{3(n-1)}}\Big) + \lambda^2\Big(\frac{k}{2}+\frac{4\sqrt{2}k^2}{\sqrt{3(n-1)}}\Big) \Big).
\]
so by choosing
\[\lambda = \Big(L - \frac{2\sqrt{2}k^2}{\sqrt{3(n-1)}}\Big)/\Big(k+\frac{8\sqrt{2}k^2}{\sqrt{3(n-1)}}\Big) \]
we obtain, for $L>\frac{2\sqrt{6}m^2}{\sqrt{n-1}}$, $1\leq k \leq m$, $m\geq4$ and $n\geq m^4$,
\begin{align*}
\mathrm{P}_n^{--} [|\tr U^k-\E\tr U^k |\geq L] &\leq 2 \exp \Big(-\frac{ \Big(1 - \frac{2\sqrt{2}k^2}{\sqrt{3(n-1)}L}\Big)^2L^2}{2k+\frac{16\sqrt{2}k^2}{\sqrt{3(n-1)}}} \Big) \leq 2 \exp \Big(-\frac{ (1 - \frac{1}{3})^2}{2+\frac{16\sqrt{2}m}{\sqrt{3(m^4-1)}}}\frac{L^2}{m} \Big) \\
&\leq 2 \exp \Big(-\frac{2}{3(3+32\frac{\sqrt{2}}{\sqrt{85}})}\frac{L^2}{m} \Big) < 2\exp(-\frac{L^2}{12m}).
\end{align*}
If $n\geq m^3$ the last upper bound is replaced by $2e^{-\frac{L^2}{20m}}$. The claim now follows by taking the union bound:
\[\mathrm{P}_{n,m}^{--} [\mathbf{X} \notin \square_L] \leq \sum_{k=1}^m \mathrm{P}_n^{--} [|\tr U^k-\E\tr U^k |\geq \frac{L}{2}] \]
which is less than $2m e^{-\frac{L^2}{48m}}$ for $n\geq m^4$ and $2m e^{-\frac{L^2}{80m}}$ for $n\geq m^3$.
\end{proof}

\begin{proof}[Proof of Theorem \ref{L1}]
We treat the case $n\geq m^4$, the other is analogous. First observe that
\[ 2\int_ {L/2}^\infty e^{-x^2/2} \frac{\d x}{\sqrt{2\pi}} = \frac{2}{\sqrt{2\pi}} e^{-L^2/8} \int_0^\infty e^{-x^2/2-Lx/2} \d x \leq \frac{2}{\sqrt{2\pi}} e^{-L^2/8} \int_0^\infty e^{-Lx/2} \d x = \frac{4}{\sqrt{2\pi}L} e^{-L^2/8} \]
whence
\begin{align}\label{gaussiandev}
\int_{\R^m\backslash \square_L}  \frac{e^{-\|\mathbf{x}\|^2/2}}{\sqrt{2\pi}^m} \d \mathbf{x}= \Big( 2\int_ {L/2}^\infty e^{-\mathbf{x}^2/2} \frac{\d \mathbf{x}}{\sqrt{2\pi}} \Big)^m \leq \Big( \frac{4}{\sqrt{2\pi}L} e^{-L^2/8}\Big)^m < e^{-mL^2/8}  
\end{align}
if e.g. $L\geq \sqrt{3}$. Now, by definition of $\Delta_{n,m}^{(1)}$, and using the Cauchy-Schwarz inequality,
\begin{align*}
    \Delta_{n,m}^{(1)} &= \Big( \int_{\square_L} + \int_{\R^m\backslash \square_L} \Big) \Big\lvert p_{n,m}^{a,b}(\mathbf{x}) - \frac{e^{-\|\mathbf{x}\|^2/2}}{\sqrt{2\pi}^m} \Big\rvert \d \mathbf{x} \\
    &\leq L^{m/2} \Delta_{n,m}^{(2)} + \int_{\R^m\backslash \square_L} p_{n,m}^{a,b}(\mathbf{x}) \d \mathbf{x} + \int_{\R^m\backslash \square_L}  \frac{e^{-\|\mathbf{x}\|^2/2}}{\sqrt{2\pi}^m}  \d \mathbf{x}.
\end{align*}
So by (\ref{gaussiandev}) and Lemma \ref{largedev}, assuming $L>\frac{2\sqrt{6}m^2}{\sqrt{n-1}}$,
\begin{equation}\label{delta1}
\Delta_{n,m}^{(1)} < L^{m/2} \Delta_{n,m}^{(2)} + 2me^{-L^2/48m} + e^{-mL^2/8} < L^{m/2} \Delta_{n,m}^{(2)} + 3me^{-L^2/48m}.    
\end{equation}
Regard the last upper bound as a function of $L$ and consider its critical point. It satisfies
\begin{equation}\label{critical}
e^{-L^2/48m} = 4m L^{\frac{m}{2}-2}\Delta_{n,m}^{(2)}
\end{equation}
and
\[ L \leq \sqrt{48 m \log \Delta_{n,m}^{(2)\ -1}} \]
if $m\geq4$ (and assuming $L\geq 1$), which gives
\[ \Delta_{n,m}^{(1)} \leq (48 m \log \Delta_{n,m}^{(2)\ -1})^\frac{m}{4} \Delta_{n,m}^{(2)} (1+\frac{12m^2}{L^2}) \leq 2(48 m \log \Delta_{n,m}^{(2)\ -1})^\frac{m}{4} \Delta_{n,m}^{(2)}  \]
if $L\geq 2\sqrt{3}m$. But this condition follows immediately from our assumption on $\Delta_{n,m}^{(2)}$: from (\ref{critical}) we see that
\[ \Delta_{n,m}^{(2)}= \frac{L^{2-\frac{m}{2}}}{4m}e^{-L^2/48m}, \]
so as a function of $L$, $\Delta_{n,m}^{(2)}$ is decreasing and therefore bounded from below by $ 3m(2\sqrt{3e}m)^{-\frac{m}{2}}$ if $L\leq 2\sqrt{3}m$. Finally observe that if $m\geq 4$, $n\geq m^4$, then
\[ \frac{2\sqrt{6}m^2}{\sqrt{n-1}}\leq \frac{32\sqrt{2}}{\sqrt{85}} <5<2\sqrt{3}m \]
which proves that our assumption was correct.
\end{proof}

We now present some special cases for which the upper bounds in Theorem \ref{L2} and \ref{L1} simplify. The numerical constants are obtained with \textit{Wolfram Mathematica}. First recall Corollary \ref{approx1}:

\begin{cor}
If $m$, $n$ satisfy the conditions in one column of the following table
\begin{table}[H]
\begin{tabular}{|l|l|l|l|l|l|l|l|} \hline
$n\geq$ & $m^4$ & $m^5$ & $m^6$ & $m^7$ & $m^8$ & $m^9$ & $m^{10}$   \\ \hline
$m\geq$ & $10^{19}$ & $1140$ & $34$ & $11$ & $6$ & $5$ & $4$  \\ \hline
\end{tabular}
\end{table}
\noindent then,
\[\Delta_{n,m}^{(2)}\leq 8m^\frac{3}{2} \sqrt{\Omega_m} \Big(\frac{m}{2}\Big)^{\frac{m}{4}} \frac{(e^{3/2}(\log m +1))^{N}}{\sqrt{N}\sqrt{\Gamma(2N+1)}} \]
and 
\[ \Delta_{n,m}^{(1)}\leq 16 m^\frac{3}{2} \sqrt{\Omega_m} (24 nm \log N)^{\frac{m}{4}} \frac{(e^{3/2}(\log m +1))^{N}}{\sqrt{N}\sqrt{\Gamma(2N+1)}}. \]
\end{cor}

\begin{proof}
We compare the last three terms in (\ref{maineq}) with the first (which asymptotically is dominant) i.e. with
\begin{equation}\label{dom}
\Big(\frac{m}{2}\Big)^{\frac{m}{4}} \frac{m^{3/2}}{N^{\frac{m+1}{2}}} \frac{(e^{3/2}(\log m +1))^{N}}{\sqrt{\Gamma(2N+1)}}.
\end{equation}
We start with the largest (asymptotically). We seek to estimate
\[ \Big( \Big(\frac{m}{2}\Big)^{\frac{m}{4}} \frac{m^{3/2}}{N^{\frac{m+1}{2}}} \frac{(e^{3/2}(\log m +1))^{N}}{\sqrt{\Gamma(2N+1)}} \Big)^{-1} m^{m/2} \exp\Big(-\frac{(1-c_1(m))^2 N^2}{4c_2(m)(m+1)^{8/3}(\log m+1)^2}\Big). \] 
Stirling's Inequality (\ref{Stirling}) and some rearranging give the upper bound
\begin{multline*}
\frac{1}{m^{3/2}}\exp \Big(\Big( \frac{1}{N} \log\Big(\frac{2N}{e^{5/2}(\log{m} +1)}\Big) + (\frac{m}{2N^2}+\frac{3}{4N^2})\log N + \frac{m}{4N^2}\log (2m) +\frac{1}{4N^2}\log (4\pi e) \\ - \frac{(1-c_1(m))^2 }{4c_2(m)(m+1)^{8/3}(\log m+1)^2} \Big) N^2 \Big)    
\end{multline*}
We see directly that the exponent becomes negative for sufficiently large $m$ if $N\geq m^3$, and we check using \textit{Wolfram Mathematica} that the requirements on $n$ and $m$ are those given in the table, in which case we obtain the simple upper bound $m^{-3/2}$.
The other two terms are treated similarly: we divide them by (\ref{dom}),
use Stirling's approximation, rearrange them as with the previous term and check that the sign of the exponent for the ranges of $m$ and $n$ in the table is always negative. We obtain the upper bound $m^{-3/2}(\sqrt{3}+1)$. 
Finally we check that (assuming simply $m\geq 4$)
\[ \frac{16}{15}e^{13/24}(e^{9/8}+1) + \frac{(\sqrt{3}+2)}{m^{3/2}}  < 8. \]

The second inequality is a consequence of Theorem \ref{L1}. To apply it we first need to check that the assumption is satisfied but that is straightforward: the fact that $\Omega_m=\pi^{m/2}/\Gamma(\frac{m}{2}+1)$ and Stirling's inequality (\ref{Stirling}) give
\[ 8m^\frac{3}{2} \sqrt{\Omega_m} \Big(\frac{m}{2}\Big)^{\frac{m}{4}} \frac{(e^{3/2}(\log m +1))^{N}}{\sqrt{N}\sqrt{\Gamma(2N+1)}} \leq \frac{8m^{5/4}}{\sqrt{2\pi }N^\frac{3}{4}} (\pi e)^\frac{m}{4} \Big(\frac{e^\frac{5}{2}(\log m+1)}{2N}\Big)^N \]
and the right-hand side is less than $3m(2\sqrt{3e}m)^{-\frac{m}{2}}$ if $N\geq m^4$ and $m\geq 4$. We obtain
\begin{align*}
\Delta_{n,m}^{(1)} &\leq 2(48 m \log \Delta_{n,m}^{(2)\ -1})^{m/4}\Delta_{n,m}^{(2)} \\
&\leq 16 m^\frac{3}{2} \sqrt{\Omega_m} \Big((48 m \log \Big(8m^\frac{3}{2} \sqrt{\Omega_m} \Big(\frac{m}{2}\Big)^{\frac{m}{4}} \frac{(e^{3/2}(\log m +1))^{N}}{\sqrt{N}\sqrt{\Gamma(2N+1)}}\Big)^{-1}\Big)^{m/4} \Big(\frac{m}{2}\Big)^{\frac{m}{4}}\frac{(e^{3/2}(\log m +1))^{N}}{\sqrt{N}\sqrt{\Gamma(2N+1)}}
\end{align*}
since $x\mapsto (\log x^{-1})^{m/4}x$ is non-decreasing for $x\in[0,e^{-m/4}]$. To complete the proof we use
\[ 8m^\frac{3}{2} \sqrt{\Omega_m} \Big(\frac{m}{2}\Big)^{\frac{m}{4}} \frac{(e^{3/2}(\log m +1))^{N}}{N^{\frac{m+1}{2}}\sqrt{\Gamma(2N+1)}} \geq N^{-N} \]
which follows again from  the fact that $\Omega_m=\pi^{m/2}/\Gamma(\frac{m}{2}+1)$ and Stirling's inequality (\ref{Stirling}).
\end{proof}

If $m$ is not sufficiently large for the assumptions of the above corollary to hold we can instead use the following.

\begin{cor}\label{approx2}
If $n\geq m^4$, $m\geq 7$, then,
\[\Delta_{n,m}^{(2)}\leq \sqrt{\Omega_m}(m^\frac{m}{2}+\epsilon)N^\frac{m}{2}\exp\Big(-\frac{C(m)N^2}{(m+1)^{8/3}(\log m+1)^2}\Big)  \]
and if $n\geq m^4$, $m\geq 27$,
\[ \Delta_{n,m}^{(1)}\leq\sqrt{\Omega_m} \frac{(48C(m)m)^\frac{m}{4}(m^\frac{m}{2}+\epsilon) N^m  }{(m+1)^{2m/3}(\log m+1)^\frac{m}{2}}\exp\Big(-\frac{C(m)N^2}{(m+1)^{8/3}(\log m+1)^2}\Big) \]
where $\epsilon<10^{-82}$ and $C(m)=\frac{(1-c_1(m)^2}{4c_2(m)}$ satisfies 
\begin{table}[H]
\begin{tabular}{|l|l|l|l|l|l|l|l|l|l|l|l|l|} \hline
$m\geq$ & $7$ & $8$ & $9$ & $10$ & $20$ & $30$ & $40$ & $50$ & $100$ & $500$ & $1000$ \\ \hline
$C(m)\geq$ & $0.052$ & $0.056$ & $0.059$ & $0.062$ & $0.077$ & $0.085$ & $0,091$ & $0.095$ & $0.106$ & $0.125$ & $0.131$ \\ \hline
\end{tabular}
\end{table}
\end{cor}

Note that $x \mapsto x^\frac{m}{4}\exp \Big(-\frac{xN^2}{(m+1)^\frac{8}{3}(\log m +1)}\Big)$, $x>0.077$, is decreasing for $n\geq m^4$, $m\geq 27$.

\begin{proof}
This time we compare each term in (\ref{maineq}) to 
\begin{equation}\label{compared}
\exp\Big(-\frac{\left(1- c_1(m)\right)^2 N^2}{4c_2(m)(m+1)^{8/3}(\log m+1)^2}\Big).    
\end{equation}
 First we divide the second term in (\ref{maineq}) by (\ref{compared}) and check that it is bounded by a small constant, more precisely by $5\cdot 10^{-83}$, for all $n\geq m^4$, $m\geq 7$ (for $m\leq 6$ it is larger than $10^{17}$). This also holds for the last term in (\ref{maineq}) divided by (\ref{compared}), which is smaller than $2^{-1022}<3\cdot 10^{-308}$. Next we consider
\begin{align*}
&\frac{16}{15}e^{13/24}(e^{9/8}+1) \Big(\frac{m}{2}\Big)^{\frac{m}{4}} \frac{m^{3/2}}{N^{\frac{m+1}{2}}} \frac{(e^{3/2}(\log m +1))^{N}}{\sqrt{\Gamma(2N+1)}}  \\
 &\leq  \frac{16}{15}e^{13/24}(e^{9/8}+1) \frac{m^{3/2}}{\sqrt{2}\pi^\frac{1}{4}} \frac{(m/2)^{\frac{m}{4}}}{N^{\frac{m}{2}+\frac{3}{4}}} \left(\frac{e^{5/2}(\log m +1)}{2N}\right)^N
 \end{align*} 
by Stirling's inequality. The upper bound divided by (\ref{compared}) is decreasing for all $N\geq m^3$, $m\geq 4$. If $m\geq 7$, it is also bounded by $2^{-1022}$. This explains how the first inequality was obtained.
For the total variation we use again Stirling's inequality and obtain
\[\Delta_{n,m}^{(2)}\leq \frac{(2 \pi e m^{-1})^\frac{m}{4}}{(\pi m)^\frac{1}{4}} (m^\frac{m}{2}+\epsilon)N^\frac{m}{2} \exp\Big(-\frac{C(m)N^2}{(m+1)^{8/3}(\log m+1)^2}\Big)  \]
which is less than $3m(2\sqrt{3e}m)^{-\frac{m}{2}}$ if $N\geq m^3$, $m\geq 27$. Hence Theorem \ref{L1} has its condition satisfied and gives
\begin{align*}
&\Delta_{n,m}^{(1)} \leq 2 (48 m \log \Delta_{n,m}^{(2)\ -1})^{\frac{m}{4}}\Delta_{n,m}^{(2)} \\
&\leq 2\sqrt{\Omega_m}(m^\frac{m}{2}+\epsilon)N^m\left( \frac{48C(m) m }{(m+1)^{8/3}(\log m+1)^2}\right)^{\frac{m}{4}}\exp\Big(-\frac{C(m)N^2}{(m+1)^{8/3}(\log m+1)^2}\Big)
\end{align*}
since $x\mapsto (\log x^{-1})^{m/4}x$ is non-decreasing for $x\in[0,e^{-m/4}]$ and our upper bound for the $L_2$-norm is greater than $\exp\left(-\frac{C(m)N^2}{(m+1)^{8/3}(\log m+1)^2}\right)$.
\end{proof}

Finally, if we only assume that $n\geq m^3$, then via computations similar to those in the last corollary we obtain

\begin{cor}\label{approx3}
If $n\geq m^3$, $m\geq 68$,
\[\Delta_{n,m}^{(2)}\leq \sqrt{\Omega_m}(m^\frac{m}{2}+0.2)N^\frac{m}{2}\exp\Big(-\frac{C(m)N^2}{(m+1)^{8/3}(\log m+1)^2}\Big)  \]
and if $n\geq m^3$, $m\geq 10^{18}$,
\[ \Delta_{n,m}^{(1)}\leq\sqrt{\Omega_m} \frac{(80C(m)m)^\frac{m}{4}(m^\frac{m}{2}+0.2) N^m  }{(m+1)^{2m/3}(\log m+1)^\frac{m}{2}}\exp\Big(-\frac{C(m)N^2}{(m+1)^{8/3}(\log m+1)^2}\Big) \]
where $C(m)=\frac{(1-c_1(m)^2}{4c_2(m)}$ is as in the previous corollary.
\end{cor}

\begin{remark}
Using these last two corollaries one can check (again with Wolfram Mathematica) that
\[ \Delta_{n,m}^{(1)}\leq N^{-0.3N} \]
if $n\geq m^4$, $m\geq 1000$, and
\[ \Delta_{n,m}^{(1)}\leq N^{-0.8\sqrt{N}} \]
if $n\geq m^3$, $m\geq 10^{19}$.
\end{remark}


\end{document}